\documentclass[twocolumn,10pt]{asme2ej}

\usepackage[cmex10]{amsmath}
\usepackage[pdftex]{graphicx}
\usepackage{epsfig}
\usepackage{pdfsync}
\usepackage{amssymb}
\usepackage{color}
\usepackage{graphics}
\usepackage{empheq}
\usepackage{hyperref}
\usepackage{array}
\usepackage{cite}
\usepackage{subfig}
\usepackage{bm}
\usepackage{pgfplots}
\usepackage{units}

\def\pa{\partial}

\def\i1n{i=1,\cdots,n}
\def\j1n{j=1,\cdots,n}
\def\ij1n{i,j=1,\cdots,n}

\newlength\figureheight
\newlength\figurewidth

\newtheorem{assum}{Assumption}

\newtheorem{coro}{Corollary}
\newtheorem{remark}{Remark}
\newtheorem{lem}{Lemma}

\newcommand{\ba}{\begin{align}}
\newcommand{\ea}{\end{align}}
\newcommand{\fr}{\frac}
\newcommand{\gm}{\gamma}

\newcommand{\alp}{\alpha}
\newtheorem{thm}{Theorem}

\title{ Stabilization of Filament Production Rate for Screw Extrusion-Based  Polymer 3D-Printing}

\author{Shumon Koga 
    \affiliation{
	Mechanical and Aerospace Engineering\\
	University of California\\
	San Diego, California 92093\\
    Email: skoga@eng.ucsd.edu
    }	
}

\author{David Straub 
    \affiliation{ Institute for System Dynamics\\
     University of Stuttgart\\
      Email: st101167@stud.uni-stuttgart.de

    }
}

\author{Mamadou Diagne  
    \affiliation{ Mechanical Aerospace and Nuclear Engineering\\ 
     Rensselaer Polytechnic Institute\\
      Troy, New York 12180 \\
      Email: diagnm@rpi.edu

    }
}
\author{Miroslav Krstic   
    \affiliation{ 
    Mechanical and Aerospace Engineering\\
	University of California\\
	San Diego, California 92093\\
    Email: krstic@ucsd.edu
    }
}

\begin{document}

\maketitle    

\begin{abstract}
{\it Polymer 3D-printing has been commercialized rapidly during recent years, however, there remains a matter of improving the manufacturing speed. Screw extrusion has a strong potential to fasten the process through simultaneous operation of the filament production and the deposition. This paper develops a control algorithm for screw extrusion-based 3D printing of thermoplastic materials through an observer-based output feedback design. We consider the thermodynamic model describing the time evolution of the temperature profile of an extruded  polymer by means of a  partial differential equation (PDE) defined on the time-varying domain. The time evolution of the spatial domain is governed by an ordinary differential equation (ODE) that reflects the dynamics of  the position of the phase change interface between polymer granules and molten polymer deposited as a molten filament. Steady-state profile of the distributed temperature along the extruder is obtained when the desired setpoint for the interface position is prescribed. To enhance the feasibility of our previous design, we develop a PDE observer to estimate the temperature profile via measured values of surface temperature and the interface position. An output feedback control law considering a cooling mechanism at the boundary inlet as an actuator is proposed. In extruders the control of raw material temperature is commonly achieved using preconditioners as part of  the inlet feeding mechanism. For some given screw speeds that correspond to slow and fast operating modes, numerical simulations are conducted to prove the performance of the proposed controller.  The convergence of the interface position to the desired  setpoint is achieved under physically reasonable temperature profiles. 
}
\end{abstract}



\section{Introduction}

On the verge of new manufacturing techniques, additive manufacturing stands out as versatile tool for high flexibility and fast adaptability in production. It is applicable in a variety of producing industries, ranging from tissue engineering \cite{mironov2003}, thermoplastics
\cite{valkanaers2013}, metal \cite{ladd2013} and ceramic
\cite{seitz2005} fabrication. One of the most popular types of 3D
printing is Fused Deposition Modeling (FDM) \cite{mohamed2015}, which
uses filaments as
raw material, that have to be precisely manufactured to achieve a good final product quality \cite{bukkapatnam2007}.

From the polymer processing and extrusion cooking industry, screw
extruders are well-known devices. Results stated in \cite{morton1989polymer,tadmor2006principles,kulshreshtha1995,li2001extruder}
give an in-depth description of screw geometrics, extruder setups
and describe the dynamics of extrusion process consisting of a  conveying zone, a melting
zone, and a mixing zone. A mathematical description of such a model is derived by mass, momentum and energy balances and appears as coupled transport equations coupled through a moving interface. This model is used in \cite{li2001extruder} to describe  an extrusion cooking process. The boundary control of  a similar model is achieved  in  \cite{diagne2015boundary,diagne2016boundary} under the assumption of constant viscosity.

More recent contributions considered screw  extrusion as a useful technology for  3D printing
application \cite{valkanaers2013,drotman2016control,diagne2017} allowing to manufacture a wider variety of materials than FDM, while using polymer granules as raw
material \cite{valkanaers2013}. In
\cite{diagne2017}, a time-delay control was developed on a model consisting of  two phases similarly to \cite{li2001extruder}. In both cases,  stabilization of the moving interface separating a conveying and  a melting zone is achieved with a fast convergence rate. Another approach which enables  to control  screw extruders in 3D
printing is proposed in
\cite{drotman2016control}, where an energy-based model is established, simplifying the implementation of the control law and circumventing difficulties with state measurement. In other words, the control of the outflow rate at  the nozzle only relies on the measurement of the heater current and the screw speed.

In the screw extrusion process, solid material is convected from the feed to the nozzle located at the end of  a heating chamber.  The solid raw material is melted and mixed before
being  expelled through the nozzle as a thin filament. For these process, the thermal behavior is an important factor which characterize final product quality. In fact, heat is supplied into the system  by the heaters  surrounding the  extruder's  barrel on the one hand and by the viscous heat  generation due to a  shearing effect  \cite{li2001extruder} on the other hand. 
The process of the phase transition from solid to liquid polymer can be described as a Stefan problem \cite{gupta}. In this context, the dynamics of the solid-liquid phase interface is derived from the energy conservation in which the latent heat required for melting
is driven by the internal heat of the liquid phase, resulting in the interface velocity to be proportional to the temperature gradients of the
adjacent phase. For instance, in \cite{escobedo2013classical}, the Stefan problem for a polymer crystallization process is described, and the analytical crystallization time is derived.

From a control perspective, the boundary stabilization of the interface position for the one-phase Stefan problem was recently developed in \cite{Shumon16}--\cite{koga_2019delay} based on the "backstepping method"  \cite{miroslav08, miroslav09}. More precisely,  in \cite{Shumon17journal}, the observer-based output feedback control was designed via a nonlinear backstepping transformation and the exponentially stabilization of the closed-loop system was proved without imposing any {\em a priori} assumption. Similarly, in \cite{koga2017backstepping} the full-state feedback control design for the one-phase Stefan problem with flowing liquid was achieved enabling to exponentially stabilize the system to a constant steady-state. The application of the backstepping method for Stefan problem to the aforementioned 3D-printing process was covered in \cite{koga2018polymer} by developing the thermodynamic model including the effects of screw speed and the barrel temperature, and designing the full-state feedback control law to stabilize the ratio between the polymer granules and melt polymer. 

This paper extends the results in \cite{koga2018polymer} by: 
\begin{itemize} 
\item $\bullet$ developing a PDE observer to estimate the temperature profile of the solid polymer granules with utilizing available measurements, 
\item $\bullet$  designing the associated observer-based output feedback control law and proving the stability of closed-loop system,   
\item $\bullet$  and verifying the performance of the designed observer and output feedback control law in numerical simulation.  	
\end{itemize}

First, the thermodynamic model of the polymer granules and melt polymer in screw extrusion is described, and the steady-state temperature profiles given by a prescribed setpoint of the interface position are analytically solved as in \cite{koga2018polymer}. Second, a PDE observer is constructed as a copy of the plant plus the measurement error of the surface temperature multiplied by a constant observer gain. The designed observer is shown to be exponentially convergent to the granular pellets' temperature profile along the heating chamber of the extruder. Third, using the designed observer, the output feedback control law for the boundary heat flux to stabilize the interface position at the desired setpoint is derived based on a similar manner to the full-state feedback design in \cite{koga2018polymer}, and the stability of the closed-loop system is proved under some realistic assumptions. Finally, simulation results are provided to illustrate the desired performance of the designed observer and the output feedback control design for some given screw speeds that correspond to slow and fast operating extrusion process. 

This paper is organized as follows. The thermodynamic model of the screw extruder is developed in Section \ref{model}, and the steady-state analysis is provided in Section \ref{ssa}. The observer design is presented in Section \ref{sec:observer}, and the associated output feedback control design is derived in Section \ref{controldesign}. The stability proof of the entire closed-loop system for a specific setup is established in Section \ref{sec:theo}. Simulation results are presented in order to analyze the controller's performance  in Section \ref{simulation}. We complete the paper in Section
\ref{conclusion} with concluding remarks.

\section{Thermodynamic Model of Screw Extruder}\label{model}

\begin{figure}[t]
\centering
\includegraphics[width=2.0in]{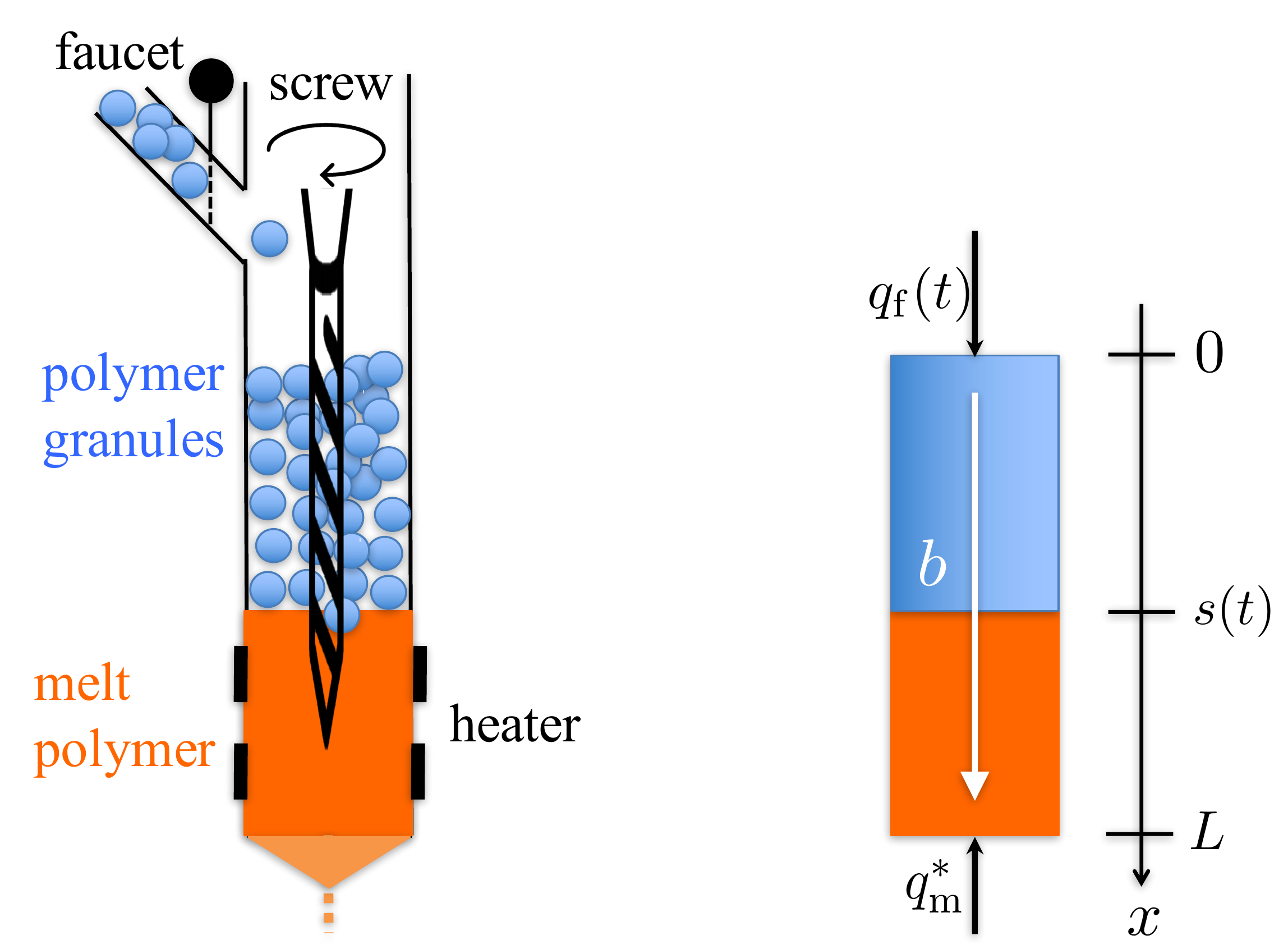}\\
\caption{Schematic of screw extruder for original description (left) and model description (right). }
\label{fig:stefan}
\end{figure}

We focus on the thermodynamic model of the screw extrusion process in one-dimensional coordinate along the vertical axis, motivated by \cite{tadmor2006principles} which developed a thermodynamic phase change model for polymer processing. The model provides the time evolution of the temperature profile of the extruded material and the interface position between the feeded polymer granules and the molten polymer. The granular pellets  are conveyed by the screw rotation at a given speed $b$ along the vertical axis while the  barrel  temperature is uniformly maintained at $T_{{\rm b}}$. Defining $T_{{\rm s}}(x,t)$ and $T_{{\rm l}}(x,t)$ as the  temperature profiles of solid phase (polymer granules) over the spatial domain $x \in (0,s(t))$  and liquid phase (molten polymer) over the spatial domain $x \in (s(t),L)$, respectively, the following thermodynamical  model
\begin{align}\label{sys1}
\fr{\pa T_{{\rm s}}}{\pa t}(x,t) =& \alp_{{\rm s}} \fr{\pa^2 T_{{\rm s}}}{\pa x^2}(x,t) - b \fr{\pa T_{{\rm s}}}{\pa x}(x,t)  \notag\\
& + h_{{\rm s}} \left( T_{{\rm b}} -  T_{{\rm s}}(x,t) \right), \textrm{for} \quad 0<x<s(t), \\
\label{sys2}\fr{\pa T_{{\rm l}}}{\pa t}(x,t) =& \alp_{{\rm l}} \fr{\pa^2 T_{{\rm l}}}{\pa x^2}(x,t) - b \fr{\pa T_{{\rm l}}}{\pa x}(x,t)  \notag\\
& + h_{{\rm l}} \left( T_{{\rm b}} - T_{{\rm l}}(x,t) \right), \textrm{for} \quad s(t)<x<L
\end{align} 
is derived  from the energy conservation and heat conduction laws. In this paper, we consider the temperature distribution in the liquid to be static as stated in \eqref{ss1} and in Assumption \ref{ass:ss} (see Section \ref{sec:error}). Here, $\alp_{i} = \fr{k_{i}}{\rho_i c_i}$ and $h_{i} = \fr{\bar{h}_{i}}{\rho_i c_i}$, where $\rho_i$, $c_i$, $k_i$, and $\bar{h}_i$ for $i\in\{s,l\}$ are the density, the heat capacity, the thermal conductivity, and the heat transfer coefficient, respectively and the subscripts $s$ and $l$ are associated to the solid or liquid phase, respectively. Referring to \cite{vergnes98} which introduces a model of spatially averaged temperature for screw extrusion, we incorporate the convective heat transfer through the barrel temperature in \eqref{sys1} \eqref{sys2}. The boundary conditions at $x=0$ and $x=L$ follow the heat conduction law, and the temperature at the interface $x=s(t)$ is maintained at the melting point $T_{{\rm m}}$, described as 
\begin{align}\label{sys3}
 \fr{\pa T_{{\rm s}}}{\pa x}(0,t) = - \fr{q_{{\rm f}}(t)}{k_{{\rm s}}} , \quad T_{{\rm s}}(s(t),t) = T_{{\rm m}}, \\
 \label{sys4} \fr{\pa T_{{\rm l}}}{\pa x}(L,t) = \fr{q_{{\rm m}}^*}{k_{{\rm l}}} , \quad T_{{\rm l}}(s(t),t) = T_{{\rm m}}, 
\end{align}
where $q_{{\rm f}}(t)<0$ is a freezing controller at the inlet and $q_{{\rm m}}^*>0$ is a heat flux at the nozzle which is assumed to be constant in time. The interface dynamics is derived by the energy balance at the interface as 
\begin{align}
\label{sys5}\rho_{s} \Delta H \dot{s}(t) = k_{{\rm s}} \fr{\pa T_{{\rm s}}}{\pa x}(s(t),t) - k_{{\rm l}} \fr{\pa T_{{\rm l}}}{\pa x}(s(t),t). 
\end{align}
The equations \eqref{sys1}-\eqref{sys5} are the solid-liquid phase change model known as "two-phase Stefan problem". Such a phase change model was developed for polymer processing 
\begin{remark} \emph{
In this paper, we assume the pressure in the chamber to be static and the melting temperature is constant to avoid supercooling. Then, to keep the physical state of each phase, the following conditions must hold:
\begin{align} \label{valid1}
T_{{\rm s}}(x,t) \leq& T_{{\rm m}}, \quad \forall x \in(0,s(t)), \quad \forall t>0, \\
\label{valid2}T_{{\rm l}}(x,t) \geq& T_{{\rm m}}, \quad \forall x \in(s(t),L), \quad \forall t>0, 
\end{align}
which represent the model validity conditions. }
\end{remark} 

\begin{remark} \emph{
We assume the existence of a heating/cooling system that maintains the  pellets  at a controlled temperature as stated in \eqref{sys3}, which describes the heat flux control at the inlet.  Extruders  can be  equipped with raw material preconditioners as  intermediate unit operators, which for instance help to pre-heat  ingredients before they enter the extruder chamber  by adding steam.  The prconditioners are usually located between the inlet and the extruder chamber and a continuous flow of material from the feeder to the preconditioner is maintained \cite{Riaz2000, Fang2003}.}
\end{remark}

\section{Steady-state and analysis} \label{ssa}
To ensure  a continuous  extrusion process, the control of the quantity of  molten polymer that  remains in the extruder chamber at any given time  is crucial. By definition, the volume of fully melted material contained in the chamber is directly related to the position of the solid-liquid interface that needs to be controlled, consequently. Physically, any given position of the interface along the spatial domain correspond to a melt temperature profile along the extruder.
  \subsection{Steady-state solution} 
An analytical solution  of the steady-state temperature profile denoted as $\left(T_{{\rm s},{\rm eq}}(x), T_{{\rm l},{\rm eq}}(x)\right)$  for any given setpoint value  of the interface position defined as $s_{{\rm r}}$, can be computed by setting the time derivative of the system \eqref{sys1}-\eqref{sys5} to zero. Hence, from  \eqref{sys1} and  \eqref{sys2} the following set of ordinary differential equations in space are obtained 
%
\begin{align}\label{eq1}
\begin{cases}
0= &\alp_{{\rm s}} T_{{\rm s},{\rm eq}}''(x) - b T_{{\rm s},{\rm eq}}'(x) 
+ h_{{\rm s}} \left( T_{{\rm b}} - T_{{\rm s},{\rm eq}}(x) \right) ,  \\
0 =& \alp_{{\rm l}} T_{{\rm l},{\rm eq}}''(x) - b T_{{\rm l},{\rm eq}}'(x) + h_{{\rm l}} \left(T_{{\rm b}} - T_{{\rm l},{\rm eq}}(x) \right), 
\end{cases}
\end{align}
and the boundary values are given as 
\begin{align}
\begin{cases}
 \label{eqbc1}T_{{\rm s},{\rm eq}}'(0) =& - \fr{q_{{\rm f}}^*}{k_{{\rm s}}} , \quad T_{{\rm s},{\rm eq}}(s_{r}) = T_{{\rm m}}, \\
 T_{{\rm l},{\rm eq}}'(L) =& \fr{q_{{\rm m}}^*}{k_{{\rm l}}} , \quad T_{{\rm l},{\rm eq}}(s_{{\rm r}}) = T_{{\rm m}}.\end{cases}
\end{align}
At equilibrium, the interface equation \eqref{sys5} satisfies the following equality
\begin{align}\label{fluxbalance}
0 =& k_{{\rm s}} T_{{\rm s},{\rm eq}}'(s_{{\rm r}}) - k_{{\rm l}} T_{{\rm l},{\rm eq}}'(s_{{\rm r}}). 
\end{align}
The solution to the set of  differential equations \eqref{eq1} has the following  form 
\begin{align}\begin{cases}
\label{ss1}
T_{{\rm l},{\rm eq}}(x) =& p_1 e^{q_1 (x-s_{{\rm r}})} + p_2 e^{q_2 (x-s_{{\rm r}})} + T_{{\rm b}}, \\
T_{{\rm s},{\rm eq}}(x) =& p_3  e^{q_3 (x-s_{{\rm r}})} + p_4  e^{q_4 (x-s_{{\rm r}})} + T_{{\rm b}} , 
\end{cases}
\end{align}
where 
\begin{align}
q_1 =& \fr{  b + \sqrt{ b^2 + 4 \alpha_{{\rm l}}  h_{{\rm l}} }}{2 \alpha_{{\rm l}}} , \quad q_2 = \fr{ b - \sqrt{ b^2 + 4 \alpha_{{\rm l}} h_{{\rm l}} }}{2 \alpha_{{\rm l}}}, \\
q_3 =& \fr{  b + \sqrt{ b^2 + 4 \alpha_{{\rm s}}  h_{{\rm s}} }}{2 \alpha_{{\rm s}}} , \quad q_4 = \fr{ b - \sqrt{ b^2 + 4 \alpha_{{\rm s}} h_{{\rm s}} }}{2 \alpha_{{\rm s}}}. 
\end{align}
Let $r = T_{{\rm b}} - T_{{\rm m}}$. Substituting \eqref{ss1} into the boundary conditions \eqref{eqbc1} and \eqref{fluxbalance},  we obtain
\begin{align}\label{p1}
p_1 =& \fr{r q_2 e^{q_2 (L- s_{{\rm r}})}+ q_{{\rm m}}^*/k_{{\rm l}} }{q_1 e^{q_1 (L- s_{{\rm r}})} - q_2 e^{q_2 (L- s_{{\rm r}})}  }, \\
\label{p2} p_2 =& - \fr{r q_1 e^{q_1 (L- s_{{\rm r}})}+ q_{{\rm m}}^*/k_{{\rm l}} }{q_1 e^{q_1 (L- s_{{\rm r}})} - q_2 e^{q_2 (L- s_{{\rm r}})}  }, \\
\label{p3} p_3 =& \fr{r q_4 + K/k_{{\rm s}}}{q_3 - q_4} , \\
\label{p4} p_4 =&  \fr{- r q_3 - K/{k_{{\rm s}}}}{q_3 - q_4}, \\
\label{K}K =& \fr{ k_{{\rm l}}r (-q_1 q_2) \left( e^{q_1 (L- s_{{\rm r}})} - e^{q_2 (L- s_{{\rm r}})} \right) + (q_1 - q_2) q_{{\rm m}}^* }{q_1 e^{q_1 (L- s_{{\rm r}})} - q_2 e^{q_2 (L- s_{{\rm r}})}  } , 
\end{align}
and the steady-state input is given by  
\begin{align}\label{ssinput}
q_{{\rm f}}^* =& p_3 q_3 e^{-q_3 s_{{\rm r}}} + p_4 q_4  e^{ -q_4 s_{{\rm r}}} . 
\end{align}
Hence, once the parameters $(s_{{\rm r}}, T_{{\rm b}}, q_{{\rm m}}^*)$ are prescribed, the steady-state input is uniquely obtained. 

\subsection{Barrel temperature condition for a valid steady-state}
For the model validity, the steady-state must satisfy \eqref{valid1} and \eqref{valid2} which restricts the barrel temperature to some physically admissible values.
\begin{lem}\label{lem-valid}
If the barrel temperature satisfies
\begin{align}
- \underline{q} \leq T_{{\rm b}} - T_{{\rm m}} \leq \bar{q}, 
\end{align}
where 
\begin{align}
\underline{q} =& \fr{(q_1 - q_2) q_{{\rm m}}^*}{ q_{den}} , \quad \bar{q} =- \fr{q_{{\rm m}}^*}{ k_{{\rm l}} q_2   e^{q_2 (L- s_{{\rm r}})}}, \\
q_{den} = & - k_{{\rm l}} q_1 q_2
 \left( e^{q_1 (L- s_{{\rm r}})} - e^{q_2 (L- s_{{\rm r}})} \right) \notag\\
&+  k_{{\rm s}} q_3 \left( q_1 e^{q_1 (L- s_{{\rm r}})} - q_2 e^{q_2 (L- s_{{\rm r}})} \right), 
\end{align}
then the steady-state solution satisfies  \eqref{valid1} and \eqref{valid2}.
\end{lem}
\begin{proof}
Since $T_{{\rm l},{\rm eq}}(s_{{\rm r}})=T_{{\rm m}}$, it is necessary to have $T_{{\rm l},{\rm eq}}'(s_{{\rm r}})\geq 0$ which yields
\begin{align}\label{cond1}
p_1 q_1 + p_2 q_2 \geq 0. 
\end{align}
Substituting \eqref{p1} and \eqref{p2} into \eqref{cond1}, we get 
\begin{align}\label{cond2}
T_{{\rm b}} - T_{{\rm m}} \geq  \fr{ \left( q_1 - q_2 \right)q_{{\rm m}}^*}{ k_{{\rm l}}q_1 q_2 \left( e^{q_1 (L- s_{{\rm r}})} - e^{q_2 (L- s_{{\rm r}})} \right)} , 
\end{align}
knowing that $q_1 q_2<0$. With the help of \eqref{cond1} and from   \eqref{ss1} the derivative of $T_{{\rm l},{\rm eq}}(x)$  satisfies $T_{{\rm l},{\rm eq}}'(x) \geq  p_1 q_1 \left( e^{q_1 (x-s_{{\rm r}})} - e^{q_2 (x-s_{{\rm r}})} \right)$. Thus, the sufficient condition of $T_{{\rm l},{\rm eq}}'(x)\geq 0$ for $\forall x\in(s_{{\rm r}}, L)$ is $p_1 q_1\geq0$ which yields 
\begin{align}\label{upper}
 T_{{\rm b}} - T_{{\rm m}}  \leq- \fr{q_{{\rm m}}^*}{ k_{{\rm l}}q_2   e^{q_2 (L- s_{{\rm r}})}}. 
 \end{align}
Next, the solid steady-state satisfies $T_{{\rm s},{\rm eq}}(s_{{\rm r}})=T_{{\rm m}}$, so it is necessary to have $T_{{\rm s},{\rm eq}}'(s_{{\rm r}})\geq 0$ leading to $p_3 q_3 + p_4 q_4 \geq 0$ which trivially holds under condition of \eqref{cond1}. 
Hence, from \eqref{ss1},  the derivative of $T_{{\rm s},{\rm eq}}(x)$   satisfies $T_{{\rm s},{\rm eq}}'(x) \geq p_4 q_4 \left( - e^{q_3 (x-s_{{\rm r}})} + e^{q_4 (x-s_{{\rm r}})} \right) $. Then, the sufficient condition for $T_{{\rm s},{\rm eq}}'(x)\geq 0$ is $p_4 q_4\geq0$, which yields 
\begin{align}\label{lower}
  T_{{\rm b}} - T_{{\rm m}}  \geq - \fr{(q_1 - q_2) q_{{\rm m}}^*}{q_{den}} . 
\end{align}
One can notice that condition  \eqref{lower} is less conservative than condition \eqref{cond2}. Hence, combining \eqref{upper} and \eqref{lower}, we conclude Lemma \ref{lem-valid}. 
\end{proof}

\section{Estimator Design of the Temperature Profile}\label{sec:observer}
Our previous work in \cite{koga2018polymer} presented the full-state feedback control law by assuming that the spatially distributed temperature profile can be measured. Some imaging-based thermal sensors such as IR camera enables to capture the entire profile of temperature, however, these sensors include high noise and detect the temperature of the chamber which contains a nominal error from the temperature of the polymer inside. Instead, single point thermal sensors such as thermocouples enable to accurately measure the surface temperature at the inlet of the extruder. Moreover, the interface position between the polymer granules and the melt polymer can be detected by cameras via image signal processing. Thus, we build an observer to estimate the temperature profile with utilizing these two available measurements. 

Let $\hat T_{{\rm s}}(x,t)$ be the estimated temperature profile. The observer design for $\hat T_{{\rm s}}(x,t)$ is stated the following theorem.   

\begin{thm} \label{thm:observer} 
Consider the plant model \eqref{sys1}, \eqref{sys3} with the two available measurements of  
\begin{align} \label{measure} 
Y_1(t) = s(t), \quad Y_2(t) = T_{s}(0,t),   
\end{align} 
and the following PDE observer 
\begin{align}\label{obs1}
\fr{\pa \hat{T}_{{\rm s}}}{\pa t}(x,t) =& \alp_{{\rm s}} \fr{\pa^2 \hat{T}_{{\rm s}}}{\pa x^2}(x,t) - b \fr{\pa \hat{T}_{{\rm s}}}{\pa x}(x,t)  \notag\\
& + h_{{\rm s}} \left( T_{{\rm b}} -  \hat{T}_{{\rm s}}(x,t) \right), \quad 0<x<Y_1(t), \\
\label{obs2} \fr{\pa \hat{T}_{{\rm s}}}{\pa x}(0,t) =& - \fr{q_{{\rm f}}(t)}{k_{{\rm s}}} - \gamma \left(Y_2(t) - \hat{T}_{{\rm s}}(0,t) \right), \\
\label{obs3}  \hat{T}_{{\rm s}}(s(t),t) =& T_{{\rm m}}, 
\end{align}
where $\gamma =\frac{b}{2 \alp_{{\rm s}}}$. Assume that $s(t) \in (0, L)$ and $\dot s(t)\geq 0$ for all $t \geq 0$. Then, the observer error system is exponentially stable at the origin in the sense of the norm 
\begin{align} \label{Phitilde} 
\tilde{\Phi}(t):= || T_{{\rm s}}(x,t) - \hat{T}_{{\rm s}}(x,t) ||_{{\cal H}_1} . 
\end{align} 
More precisely, there exists a positive constant $\tilde M>0$ such that the following inequality holds: 
\begin{align} \label{normdecay} 
\tilde{\Phi}(t) \leq \tilde M \tilde \Phi(0) e^{-2 	\left(h_s + \frac{b^2}{4 \alpha_{{\rm s}}}  + \frac{\alpha_{s}}{4 L^2}\right) t} 
\end{align}

\end{thm} 
\begin{proof} 

Let $\tilde u$ be the estimation error state defined by 
\begin{align} \label{utildedef} 
\tilde u := T_{{\rm s}} - \hat{T}_{{\rm s}}.
\end{align}
Subtraction of the observer system \eqref{obs1}--\eqref{obs3} from the plant \eqref{sys1} and \eqref{sys3} yields the following estimation error system: 
\begin{align}\label{estimationerr1}
\tilde{u}_{t}(x,t) =& \alp_{{\rm s}} \tilde{u}_{xx}(x,t) - b \tilde{u}_{x}(x,t) \notag\\
&  - h_{{\rm s}}\tilde{u}(x,t),  \quad 0<x<s(t), \\
 \tilde{u}_{x}(0,t) =& \gamma\tilde{u}(0,t), \\
\label{estimationerr3}\tilde{u}(s(t),t) =& 0.  
\end{align}
Let us introduce the following change of variable
\begin{align} \label{ztildedef} 
\tilde z(x,t) = \tilde u(x,t) e^{-\gamma x}	. 
\end{align}
Then, $\tilde u$-system in \eqref{estimationerr1}--\eqref{estimationerr3} is converted into the following $\tilde z$-system: 
\begin{align} \label{ztilde1}
\tilde z_t =& \alpha \tilde z_{xx} - \lambda \tilde z, \\
\tilde z_x(0,t) =& 0, \\
\tilde z(s(t),t) =& 0. 	\label{ztilde3} 
\end{align}
where $\lambda =  h_s + \frac{b^2}{4 \alpha_{{\rm s}}}$. To study the stability of the estimation error state at the origin, we consider the Lyapunov functional  
\begin{align} \label{defV1tilde} 
\tilde V = \frac{1}{2} || \tilde{z} ||_{{\cal H}_1}^2= \frac{1}{2}  \int_0^{s(t)}\tilde{z}(x,t)^2 dx + \frac{1}{2}  \int_0^{s(t)}\tilde{z}_{x}(x,t)^2 dx . 
\end{align} 
Taking the time derivative of \eqref{defV1tilde} along the solution of \eqref{estimationerr1}--\eqref{estimationerr3} leads to 
\begin{align}
\dot{\tilde V} =& \frac{\dot{s}(t)}{2} \tilde{z}(s(t),t)^2 + \int_0^{s(t)}\tilde{z}(x,t) \tilde{z}_{t}(x,t) dx  \notag\\
&+ \frac{\dot{s}(t)}{2} \tilde{z}_{x}(s(t),t)^2+ \int_0^{s(t)} \tilde z_x(x,t) \tilde z_{xt}(x,t) dx \notag\\
=&  \int_0^{s(t)}\tilde{z}(x,t) \left(  \alp_{{\rm s}} \tilde{z}_{xx}(x,t)   - \lambda \tilde{z}(x,t) \right) dx + \frac{\dot{s}(t)}{2} \tilde{z}_{x}(s(t),t)^2 \notag\\
&  + \tilde z_x(x,t) \tilde z_{t}(s(t),t) - \tilde z_x(0,t) \tilde z_{t}(0,t)  \notag\\
&- \int_0^{s(t)} \tilde z_{xx}(x,t) \tilde z_{t}(x,t) dx . \label{Vtilde1ineq} 
\end{align} 
Note that taking the total time derivative of the boundary condition \eqref{ztilde3} yields $\tilde z_t(s(t),t) = - \dot{s}(t) \tilde z_{x}(s(t),t)$. Substituting this into \eqref{Vtildeineq} and taking the integration by parts, we get
\begin{align} 
\dot{\tilde V} =& - \alp_{{\rm s}} || \tilde{z}_{xx} ||_{L_2}^2 - (\alp_{{\rm s}} + \lambda)  ||\tilde{z}_x||_{L_2}^2 \notag\\
& - \lambda ||\tilde{z}||_{L_2}^2  - \frac{\dot{s}(t)}{2}\tilde{z}_{x}(s(t),t)^2. \label{Vtildeineq} 
\end{align} 
With the help of $s(t) \in (0,L)$, Poincare's inequality gives $|| \tilde z||_{L_2}^2 \leq 4 L^2 || \tilde z_{x}||_{L_2}^2$ and $|| \tilde z_{x}||_{L_2}^2 \leq 4 L^2 || \tilde z_{xx}||_{L_2}^2$. Applying these inequalities and $\dot s(t) \geq 0$ to \eqref{Vtildeineq} leads to the following differential inequality
\begin{align} \label{Vtildeineq2} 
\dot{\tilde{V}} \leq 	& - 2 \left(\lambda + \frac{\alpha_{s}}{4 L^2} \right) \tilde{V} . 
\end{align}
Applying the comparison principle to \eqref{Vtildeineq2} yields 
\begin{align} 
\tilde V (t) \leq \tilde V(0) e^{- 2 \left(\lambda + \frac{\alpha_{s}}{4 L^2} \right) t} 	. \label{Vtildefin} 
\end{align}
By the definition of $\tilde z$ given in \eqref{ztildedef}, for the norm of $\tilde u$-system, the following upper and lower bounds hold
\begin{align} 
|| \tilde z ||_{L_2}^2 \leq || \tilde u||_{{L_2}}^2 \leq e^{2 \gamma L}	|| \tilde z ||_{L_2}^2, \\
|| \tilde z_x ||_{L_2}^2 \leq 2 || \tilde u_{x} ||_{{L_2}}^2 + 2 \gamma ^2 || \tilde u||_{{L_2}}^2 \notag\\
|| \tilde u_x||_{{L_2}}^2 \leq 2 e^{2 \gamma L}	(|| \tilde z_x ||_{L_2}^2 +  \gamma^2 || \tilde z ||_{L_2}^2). 
\end{align}
Hence, by defining $\tilde \Phi(t) = || \tilde u ||_{{\cal H}_1}^2$, the following inequalities hold 
\begin{align} \label{norm_ineq} 
\tilde M_1 \tilde V \leq  \tilde \Phi \leq \tilde M_2 \tilde V	
\end{align}
where $\tilde M_1 = 1/\max \{3, 2 \gamma^2\}$, and $\tilde M_2 = e^{2 \gamma L} \max \{3, 2 \gamma^2\} $. Applying \eqref{Vtildefin} to \eqref{norm_ineq} with defining $\tilde M = \tilde M_2/ \tilde M_1$ leads to the conclusion in Theorem \ref{thm:observer}. 

\end{proof} 

In addition, the estimated temperature can maintain not greater value than the true temperature in the plant, as stated in the following lemma. 
\begin{lem} \label{lem:negative} 
If $\tilde{u}(x,0) \geq 0$, $\forall x \in (0, s_0)$, then 
\begin{align} 
\tilde{u}(x,t) \geq & 0, \quad \forall x\in (0,s(t)), \quad \forall t\geq 0	, \\
\tilde{u}_{x}(s(t),t) \leq & 0, \quad \forall t\geq 0
\end{align}

\end{lem}
\begin{proof}
Applying Maximum principle to $\tilde z$-system governed by \eqref{ztilde1}--\eqref{ztilde3} leads to the statement that if $\tilde z(0,t) \geq 0$, $ \forall x \in (0, s_0)$ then $\tilde z(x,t) \geq 0$,  $\forall x\in (0,s(t)), \forall t\geq 0$. By the relation between $\tilde z$ and $\tilde u$ given in \eqref{ztildedef}, we prove Lemma \ref{lem:negative}, with the help of Hopf's lemma. 
\end{proof} 

The properties in Lemma \ref{lem:negative} are required to guarantee the positivity of the boundary heat input under the output feedback control design which is given in the later sections. 

\begin{remark} \emph{
	The convergence speed of the designed observer is characterized by $ h_s + \frac{b^2}{4 \alpha_{{\rm s}}} + \frac{\alpha_{s}}{4 L^2}$ as seen in the estimate of the norm \eqref{normdecay}, which cannot be choosen arbitrary fast for given physical constants and the manufacturing speed. The performance improvement to fasten the observer's convergence can be achieved by adding the measurement error injection to the observer PDE formulated by 
	\begin{align} 
		\fr{\pa \hat{T}_{{\rm s}}}{\pa t}(x,t) =& \alp_{{\rm s}} \fr{\pa^2 \hat{T}_{{\rm s}}}{\pa x^2}(x,t) - b \fr{\pa \hat{T}_{{\rm s}}}{\pa x}(x,t)+ h_{{\rm s}} \left( T_{{\rm b}} -  \hat{T}_{{\rm s}}(x,t) \right) \notag\\
		&  + p(x,t) (Y_2(t) - \hat{T}_{{\rm s}}(0,t)), \quad 0<x<Y_1(t), \label{PDE_obsv_fast} 
	\end{align}
where the distributed observer gain $p(x,t)$ can be designed using backstepping method as developed in \cite{Shumon17journal,koga2017battery,koga_2019seaice}. However, with the PDE observer \eqref{PDE_obsv_fast}, it is challenging to ensure the positivity of the output feedback control law. Since this paper's primary focus is on control design, we use the PDE observer given in \eqref{obs1}--\eqref{obs3}. }
	 
\end{remark}

\section{Control Design of Boundary Heat} \label{controldesign}
When the solid pellets are injected   and heated into the extruder chamber,  the amount of the molten polymer expands reducing the quantity of solid material into the chamber. Thus a cooling effect arising from the continuous feeding of cooler  pellets enables to maintain the interface at the desired setpoint. The setpoint open-loop boundary heat control  $q_{{\rm f}}(t) = q_{{\rm f}}^*$ (see \eqref{eqbc1}) is  not sufficient to drive the solid-liquid interface position to the desired setpoint. In this section, we develop the control design of the boundary heat at the inlet to drive the interface to the setpoint while stabilizing the temperature profile at the steady-state. 

\subsection{Reference error system for a dynamics reduced to a single phase }\label{sec:error}
First, we impose the following assumption on the liquid temperature. 
\begin{assum}\label{ass:ss}
The liquid temperature is at steady-state profile, i.e. $ T_{{\rm l}}(x,t)= T_{{\rm l},{\rm eq}}(x)$. 
\end{assum}

Under Assumption \ref{ass:ss}, the two-phase dynamics governed by \eqref{sys1}--\eqref{sys5} is reduced to a single-phase model. Let $(u(x,t), \hat u(x,t), X(t))$ be the reference error variables defined by
\begin{align}\label{udef} 
 u(x,t) =& - k_{{\rm s}} (T_{{\rm s}}(x,t) -  T_{{\rm s},{\rm eq}}(x)), \\
\hat u(x,t) =& - k_{{\rm s}} (\hat T_{{\rm s}}(x,t) - T_{{\rm s},{\rm eq}}(x)), \label{uhatdef} \\
 X(t) =& s(t) - s_{{\rm r}}.   
\end{align}
Note that the negative signs are included in \eqref{udef} and \eqref{uhatdef} to make the states $(u,\hat u)$ have positivity properties for the model validity conditions to hold, which is consistent with the analysis in \cite{Shumon17journal}. Then, the estimation error state $\tilde u$ defined by \eqref{utildedef} yields 
\begin{align}
\tilde u(x,t) = \hat u(x,t) - u(x,t).
\end{align}   
We rewrite the original system \eqref{sys1}--\eqref{sys5} using the reference and estimation error states $(\hat u,X,\tilde u)$. Substituting $x = s(t)$ into \eqref{uhatdef} with the help of \eqref{obs3}, we get 
\begin{align}\label{linear1}
\hat u(s(t),t) =& k_{{\rm s}} ( T_{{\rm s},{\rm eq}}(s(t)) - T_{{\rm m}}). 
\end{align}
In addition, rewriting \eqref{sys5} in term of  $\hat u(x,t)$ with $\tilde u(x,t)$ leads to the following equation of interface dynamics
\begin{align}\label{linear2}
\dot{X}(t) =& - \bar{\beta} \left( \hat u_{x}(s(t),t) - \tilde u_{x}(s(t),t) \right) \notag\\
&+ \bar{\beta}  \left( k_{{\rm s}}T_{{\rm s},{\rm eq}}'(s(t)) -  k_{{\rm l}} T_{{\rm l},{\rm eq}}'(s(t)) \right), 
\end{align}
where $\bar{\beta} = \left(\rho_{s} \Delta H\right)^{-1}$. Taking a linearization of the right hand side of \eqref{linear1} and \eqref{linear2} with respect to $s(t)$ around the setpoint $s_{{\rm r}}$ and by the steady state solutions in  \eqref{ss1}, the dynamics of the reference error system is obtained by
\begin{align}\label{err1}
 \hat u_{t}(x,t) =& \alp_{{\rm s}} \hat u_{xx}(x,t) - b \hat u_{x}(x,t) - h_{{\rm s}} \hat u(x,t), \\
\label{err2} \hat u_{x}(0,t) =& -U(t) + \gamma \tilde{u}(0,t), \\
 \hat u(s(t),t) =& C X(t), \\
 \label{err3} \dot{X}(t) =& A X(t) - \bar{\beta} \hat u_{x}(s(t),t) + \bar{\beta} \tilde u_{x}(s(t),t)  , 
  \end{align}
  where
  \begin{align}
  U(t) =& - (q_{{\rm f}}(t) - q_{{\rm f}}^*), \\
  \label{Cdef}
  C =& k_{{\rm s}} \left( p_3 q_3 + p_4 q_4 \right), \\
 \label{Adef} A =& \bar{\beta} \left( k_{{\rm s}} ( p_3 q_3^2 + p_4 q_4^2 ) - k_{{\rm l}} ( p_1 q_1^2 + p_2 q_2^2 ) \right). 
  \end{align}

\subsection{Backstepping transformation} 
A well-known design method of the output feedback control for PDEs is achieved by introducing the backstepping transformation which maps the observer PDE with using the gain kernel function derived for the full-state feedback control. Therefore, we consider the following transformation 
\begin{align}\label{bkst}
\hat w(x,t) = & \hat u(x,t) - \frac{\bar{\beta}}{\alpha_{{\rm s}}}  \int_{x}^{s(t)} \phi(x-y) \hat u(y,t) dy \notag\\
&-  \phi(x-s(t)) X(t), 
\end{align}
where $\phi$ is the gain kernel function derived in \cite{koga2018polymer}, which satisfies the following differential equation with the initial condition 
\begin{align}\label{gainODE}
\alpha_{{\rm s}}  \phi''(x) - &(b + \bar{\beta} C)  \phi'(x) - \left(A-\fr{\bar{\beta} b}{\alpha_{{\rm s}}}C + h_{{\rm s}}\right) \phi(x) =0, \\
\label{gainIC}\phi(0) =& 0, \quad \phi'(0) = \fr{c}{\bar{\beta}}, 
\end{align}
where $c>0$ is a control gain. The solution to \eqref{gainODE} with \eqref{gainIC} is uniquely given by  
\begin{align} \label{phisol} 
\phi(x) =  \fr{c}{\bar{\beta} (d_1 - d_2)} \left( e^{d_1 x} - e^{d_2 x} \right) , 
\end{align}
where 
$d_1$, $d_2$ are defined by 
\begin{align}
d_1 = &\fr{ \bar b + \sqrt{ D}}{2 \alpha_{{\rm s}}}, \quad d_2 = \fr{ \bar b - \sqrt{ D}}{2 \alpha_{{\rm s}}}, \\
\bar b =& b + \bar{\beta} C, \\
D =& \bar b^2 + 4 \alpha_{{\rm s}} \left(A-\fr{\bar{\beta} b}{\alpha_{{\rm s}}}C + h_{{\rm s}}\right) . 
\end{align} 
The full-state feedback control law developed in \cite{koga2018polymer} is given by 
\begin{align}\label{Ufull}
U_{{\rm full}}(t) =& - \gm u(0,t) - \frac{\bar{\beta}}{\alpha_{{\rm s}}} \int_{0}^{s(t)} f(x) u(x,t) dx \notag\\
& -  f(s(t))X(t),
\end{align}
where
\begin{align}
\gm = &\fr{b}{2\alpha_{{\rm s}}}, \\
f(x) = & \phi'(-x) - \gm \phi(-x), \\
=& \fr{c}{\bar{\beta} (d_1 - d_2)}  \left( (d_1 - \gm ) e^{- d_1 x} - (d_2 - \gm ) e^{ - d_2 x} \right).  
\end{align}
The associated output feedback control law is normally designed by replacing the plant state in the full-state feedback control law with the observer state. Since $X(t)$ in \eqref{Ufull} can be directly measured and its observer state is not constructed, we keep the term $X(t)$. Moreover, for the sake of proving the positivity of the designed control law later, we also hold the boundary value term $u(0,t)$ in \eqref{Ufull}, which can also be directly measured. Hence, the resulting observer-based output feedback control law is designed by 
\begin{align}\label{Ucont}
U(t) =& - \gm u(0,t) - \frac{\bar{\beta}}{\alpha_{{\rm s}}} \int_{0}^{s(t)} f(x) \hat u(x,t) dx \notag\\
& -  f(s(t))X(t),
\end{align}
Then, taking the derivatives of \eqref{bkst} in $x$ and $t$ along the solution of \eqref{err1}-\eqref{err3} with the gain kernel function \eqref{phisol}, the transformed $(\hat w, X)$-system (so-called "target system") is described by the following dynamics
\begin{align}\label{tarPDE}
 \hat w_{t}(x,t) =& \alp_{{\rm s}} \hat w_{xx}(x,t) - b \hat w_{x}(x,t) - h_{{\rm s}} \hat w(x,t) \notag\\
&+  \dot{s}(t) g(x-s(t)) X(t)\notag\\
& - \bar{\beta} \phi(x-s(t)) \tilde{u}_{x}(s(t),t), \quad 0<x<s(t) \\
\label{tarBC1} \hat w_{x}(0,t) =& \gm \hat w(0,t), \\
\label{tarBC2} \hat w(s(t),t) =& C X(t), \\
 \label{tarODE} \dot{X}(t) =& \left(A - c \right) X(t) - \bar{\beta} \hat w_{x}(s(t),t) + \bar{\beta} \tilde u_{x}(s(t),t), 
  \end{align}
where 
\begin{align} 
g(x) = &\phi'(x) - \frac{\bar{\beta}}{\alpha_{{\rm s}}} C \phi(x). 
\end{align} 
%
%
Rewriting the control law \eqref{Ucont} with respect to the boundary heat control $q_{{\rm f}}(t)$, the estimated temperature $\hat T_{{\rm s}}$, the reference steady-state $T_{{\rm s},{\rm eq}}$, and the measured variables $Y_1(t)$ and $Y_2(t)$, the resulting output feedback control is described by 
\begin{align}\label{qfcont}
q_{{\rm f}}(t) =& q_{{\rm f}}^* - \gm k_{{\rm s}} ( Y_2(t) - T_{{\rm s},{\rm eq}}(0)) \notag\\
&- \frac{\bar{\beta}k_{{\rm s}}}{\alpha_{{\rm s}}} \int_{0}^{Y_1(t)} f(x) (\hat T_{{\rm s}}(x,t) - T_{{\rm s},{\rm eq}}(x)) {\rm d}x \notag\\
& +  f(Y_1(t))(Y_1(t) - s_{r} ). 
\end{align}

\section{Theoretical Analysis for a Specific Setup}\label{sec:theo}
While the controller is designed through the backstepping method, the stability of the target system is not proven theoretically. Moreover, the condition of model validity needs to be satisfied under the control law. To achieve a theoretical result,  in this section we impose following assumptions.
\begin{assum} \label{initial-estimate}
The initial condition of the estimated temperature profile is not greater than that of the true temperature profile, i.e., 
\begin{align} 
\hat T_{{\rm s}}(x,0) \leq T_{{\rm s}}(x,0), \quad \forall x \in (0, s_0), 
\end{align}
 	where $s_0 := s(0)$.
\end{assum}

\begin{assum}\label{setup}
The barrel temperature is set as melting temperature and the external heat input is zero, i.e.
\begin{align}\label{ass-1}
T_{{\rm b}} = T_{{\rm m}}, \quad q_{{\rm m}}^* = 0. 
\end{align}
\end{assum}

\begin{coro} \label{coro:2} 
Under Assumption \ref{initial-estimate}, it holds $\tilde u(x,t) \geq 0$ and $\tilde u_x(s(t),t) \leq 0$, $\forall x \in (0,s(t)), \forall t \geq 0$, as proven in Lemma \ref{lem:negative}. 	
\end{coro}

\begin{coro}
Under Assumption \ref{setup}, the steady state profiles \eqref{ss1}, and steady state input \eqref{ssinput} becomes $T_{{\rm l},{\rm eq}}(x) = T_{{\rm m}}$, $T_{{\rm s},{\rm eq}}(x) = T_{{\rm m}}$, and $q_{{\rm f}}^* = 0$. Also, $C = 0$ and $A = 0$.  
\end{coro}

In addition, the following setpoint restriction is given. 
\begin{assum} \label{ass:setpoint} 
The setpoint is chosen to satisfy 
\begin{align}
s_{r} > s_0 + \frac{\bar{\beta}k_{{\rm s}}}{\alpha_{{\rm s}}} \int_{0}^{s_0} \fr{f(x)}{f(s_0)} (T_{{\rm m}} - \hat T_{{\rm s}}(x,0)) {\rm d}x.  
\end{align} 
\end{assum}

The main theorem is stated as following. 
\begin{thm}\label{theo}
Let Assumptions \ref{ass:ss}--\ref{ass:setpoint} hold. Then, the closed-loop system consisting of the plant \eqref{sys1}--\eqref{sys5}, the measurements \eqref{measure}, the observer \eqref{obs1}--\eqref{obs3}, and the control law \eqref{qfcont} satisfies the conditions for model validity \eqref{valid1}, \eqref{valid2}, and is exponentially stable at the origin in the norm
\begin{align} 
\hat \Phi(t) := &|| T_{{\rm s}}(x,t) - T_{{\rm s,eq}}(x) ||_{{\cal H}_1} \notag\\
&+ || T_{{\rm s}}(x,t) - \hat{T}_{{\rm s}}(x,t) ||_{{\cal H}_1}  + |s(t) - s_{{\rm r}}|  , 
\end{align}
namely, there exists a positive constant $\hat {M}>0$ such that the following estimate of the norm holds 
\begin{align}  
\hat \Phi(t) \leq \hat{M} \hat \Phi(0) e^{- d t} , 
\end{align} 
where $d = \min\left\{ \frac{\alp_{{\rm s}}}{16 s_{{\rm r}}}  + \frac{b^2}{4 \alp_{{\rm s}}} + h_{{\rm s}},	  c\right\}$. 
\end{thm}

The proof of Theorem \ref{theo} is established by showing that \eqref{valid1} and \eqref{valid2} are satisfied and employing a Lyapunov analysis through the remaining of this section. 
 
\subsection{Model validity condition}
Let $Z(t)$ be defined as
\begin{align}
Z(t) =& U(t) + \gm u(0,t) \notag\\
\label{zcontroller}=& - \frac{\bar{\beta}}{\alpha_{{\rm s}}} \int_{0}^{s(t)} f(x) \hat{u}(x,t) {\rm d}x -  f(s(t))X(t). 
\end{align}
The following lemma is stated. 
\begin{lem} \label{physical}
The following properties hold:
\begin{align} \label{Zpositive}
Z(t)>&0 , \quad \forall t \geq 0, \\
\label{property2}u(x,t) >&0, \quad \dot{s}(t) >0 \quad \forall x \in (0,s(t)), \quad \forall t \geq 0, \\
\label{position} s(0)<&s(t)<s_{{\rm r}}, \quad \forall t \geq 0. 
\end{align}
\end{lem}
\begin{proof} 
Taking the time derivative of \eqref{zcontroller} along the solution of \eqref{err1}--\eqref{err3}, we have 
\begin{align} 
\dot{Z}(t) = & - c Z(t)  -\dot{s}(t) \left(\frac{\bar{\beta}}{\alpha_{{\rm s}}} f(s(t))C + f'(s(t)) \right) X(t) \notag\\
&+ \bar \beta f(s(t)) \tilde{u}_{x}(s(t),t) \notag\\
& + \left( \frac{\bar \beta}{\alpha_{{\rm s}}} \{ \alpha_{{\rm s}}f'(s(t)) + b f(s(t))\} C - f(s(t))A  \right) X(t)\notag\\
& -  \frac{\bar \beta}{\alpha_{{\rm s}}} \{ \alpha_{{\rm s}}f'(0) + (b - \alpha_{{\rm s}} \gm) f(0)\} \hat u(0,t) \notag\\
& -  \frac{\bar \beta}{\alpha_{{\rm s}}}\int_{0}^{s(t)} ( \alpha_{{\rm s}} f''(x) + b f'(x) - h_{{\rm s}} f(x)) \hat u(x,t)  {\rm d}x .  \label{Zdot1} 
\end{align}
Taking into account $A = C = 0$, the differential equation for $\phi$ in \eqref{gainODE} is given by
\begin{align} 
	 \alpha_{{\rm s}} \phi''(x) - b \phi'(x) - h_{{\rm s}} \phi(x) = 0. 
\end{align}
Thus, recalling $ f(x) = \phi'(-x) - \gm \phi(-x)$, it holds that  
\begin{align} 
	&\alpha_{{\rm s}} f''(x) + b f'(x) - h_{{\rm s}} f(x)\notag\\
	=&\left( \alpha_{{\rm s}} \phi'''(-x) - b \phi''(-x) - h_{{\rm s}} \phi'(-x) \right) \notag\\
	&- \gm \left( \alpha_{{\rm s}} \phi''(-x) - b \phi'(-x) - h_{{\rm s}} \phi(-x) \right) \notag\\
	=& 0 . \label{integ0} 
\end{align}
Moreover, we have 
\begin{align} 
	\alpha_{{\rm s}}f'(0) + (b - \alpha_{{\rm s}} \gm) f(0) = 0. \label{boundary0} 
\end{align}
Substituting \eqref{integ0}, \eqref{boundary0}, and $A = C = 0$ into \eqref{Zdot1}, we obtain 
\begin{align}
\dot{Z}(t) =& - cZ(t) - \dot{s}(t)  f'(s(t))X(t) \notag\\
&  - \bar \beta f(s(t)) \tilde{u}_{x}(s(t),t), \label{contODI}\\
\geq & - cZ(t) - \dot{s}(t)  f'(s(t))X(t), \quad \forall t \geq 0,\label{contODE}
\end{align}
where we used Corollary \ref{coro:2} and $f(x)>0$ for the derivation from \eqref{contODI} to \eqref{contODE}. 

We prove \eqref{Zpositive} by contradiction approach. Assume that \eqref{Zpositive} is not valid, which implies $\exists t^*>0$ such that 
\begin{align}\label{assum}
Z(t) >0, \quad \forall t\in (0,t^*), \quad \quad Z(t^*) = 0. 
\end{align} 
Similarly to Lemma \ref{lem:negative}, by Maximum principle and Hopf's lemma, we get 
\begin{align} \label{contradict0} 
u(x,t)>0, \quad  \dot{s}(t) > 0, \quad \forall x\in(0,s(t)), \quad \forall t\in(0,t^*), 
\end{align}
which, with the help of Lemma \ref{lem:negative}, leads to 
\begin{align} \label{contradict1} 
\hat{u}(x,t) >& 0, \quad \forall x\in(0,s(t)), \quad \forall t\in(0,t^*), \\
s(t) >& s_0>0, \quad \forall t\in(0,t^*). \label{contradict2} 
\end{align} 
Applying \eqref{assum}, \eqref{contradict1}, and \eqref{contradict2} to \eqref{zcontroller} with $f(x)>0$ leads to 
\begin{align} 
X(t)<0, \quad \forall t\in(0, t^*). 	\label{contradict3} 
\end{align}
Therefore, applying \eqref{contradict0} and \eqref{contradict3} to \eqref{contODE} leads to 
\begin{align} \label{contradict4} 
\dot{Z}(t) > -cZ(t), \quad \forall t\in(0,t^*). 
\end{align}
Applying Gronwall's inequality to \eqref{contradict4} leads to the inequality regarding the solution of the differential equation, namely,  
\begin{align}
Z(t) \geq Z(0) e^{-ct}, \quad \forall t\in(0,t^*]. 
\end{align} 
Thus, we have $Z(t^*) \geq Z(0) e^{-ct^*} >0$ , which contradicts with the assumption \eqref{assum}. Hence, \eqref{Zpositive} is proved. Then, by Maximum principle, \eqref{property2} holds. Imposing \eqref{Zpositive} and \eqref{property2} on \eqref{zcontroller},  we obtain $X(t)<0$ which leads to \eqref{position}. 
\end{proof} 

\subsection{Stability analysis}
Taking into account $A = C = 0$, we study the stability of the target system, 
\begin{align}\label{tarPDEtheo}
 \hat w_{t}(x,t) =& \alp_{{\rm s}} \hat w_{xx}(x,t) - b \hat w_{x}(x,t) - h_{{\rm s}} \hat w(x,t) \notag\\
&+  \dot{s}(t) g(x-s(t)) X(t) \notag\\
&- \bar{\beta} \phi(x-s(t)) \tilde{u}_{x}(s(t),t), \quad 0<x<s(t) \\
\label{tarBC1theo} \hat w_{x}(0,t) =& \gm \hat w(0,t), \\
\label{tarBC2theo} \hat w(s(t),t) =& 0, \\
 \label{tarODEtheo} \dot{X}(t) =&  - c X(t) - \bar{\beta} \hat w_{x}(s(t),t) + \bar{\beta} \tilde u_{x}(s(t),t). 
  \end{align}
  Let $\hat z$ be a variable defined by 
  \begin{align} 
  \hat z(x,t) = \hat w(x,t) e^{- \gm x}.  
  \end{align} 
Recalling $\tilde z := \tilde u e^{-\gm x}$, we have the following $(\hat z, X)$-system 
  \begin{align} 
  \label{zPDE}
\hat z_{t}(x,t) =& \alp_{{\rm s}} \hat z_{xx}(x,t) - \lambda \hat z(x,t) +  \dot{s}(t) g(x-s(t)) X(t) e^{-\gm x}\notag\\
& - \bar{\beta} \phi(x-s(t))  \tilde{z}_{x}(s(t),t) , \quad 0<x<s(t) \\
\label{zBC1} \hat z_{x}(0,t) =& 0, \\
\label{zBC2} \hat z(s(t),t) =& 0, \\
 \label{zODE} \dot{X}(t) =&  - c X(t) - \bar{\beta} \hat z_{x}(s(t),t) e^{\gm s(t)} + \bar{\beta} \tilde z_{x}(s(t),t) e^{\gm s(t)}, 
  \end{align}
  where $\lambda :=  h_{{\rm s}} + \frac{b^2}{4 \alp_{{\rm s}}}$. Consider the following functional  
\begin{align}  \label{V1def}
\hat V_1 = \frac{1}{2} \int_0^{s(t)} \hat  z(x,t)^2 {\rm d}x 	. 
\end{align}
Taking the time derivative of \eqref{V1def} along the solution of \eqref{zPDE}--\eqref{zBC2} leads to 
\begin{align} 
	\dot{\hat{V}}_1 =& - \alp_{{\rm s}} || \hat z_{x}||^2  - \lambda ||\hat z||^2  \notag\\
	& + \dot{s}(t) X(t)\int_0^{s(t)} \hat z(x,t) e^{-\gm x} g(x-s(t))  {\rm d}x \notag\\
	&- \bar{\beta}  \tilde{z}_{x}(s(t),t) \int_0^{s(t)} \hat z(x,t) \phi(x-s(t)) {\rm d}x. \label{V1timeder}
\end{align}
Applying Young's and Cauchy-Schwarz inequalities to the last two lines in \eqref{V1timeder}, we get 
\begin{align} 
 &\dot{s}(t) X(t)\int_0^{s(t)} \hat z(x,t) e^{-\gm x}  g(x-s(t))  {\rm d}x \notag\\
   \leq &  \frac{\dot{s}(t)}{2} \left( X(t)^2 + || g||^2 \cdot || \hat z ||^2  \right) , \label{YCS1} 
   \end{align} 
   \begin{align} 
& - \bar{\beta}  \tilde{z}_{x}(s(t),t) \int_0^{s(t)} \hat z(x,t)   \phi(x-s(t)) {\rm d}x \notag\\
  \leq &\frac{ \bar{\beta}^2 || \phi ||^2}{2 h_{{\rm s}}}  \tilde{z}_{x}(s(t),t)^2 + \frac{h_{{\rm s}}}{2}  ||\hat z||^2 . \label{YCS2}
 \end{align} 
 Thus, applying \eqref{YCS1} and \eqref{YCS2} to \eqref{V1timeder}, one can obtain  
 \begin{align} 
 \dot{\hat{V}}_1 \leq & - \alp_{{\rm s}} || \hat z_{x}||^2  - \lambda  || \hat z||^2  \notag\\
	& +\frac{\dot{s}(t)}{2} \left( X(t)^2 + || g||^2 \cdot || \hat z ||^2  \right)  \notag\\
	&+ \frac{ \bar{\beta}^2 || \phi ||^2}{2 h_{{\rm s}}}  \tilde{z}_{x}(s(t),t)^2. \label{V1fin} 
\end{align} 
Consider the following functional  
\begin{align} \label{V2def}
 \hat V_2 = \frac{1}{2} \int_0^{s(t)} \hat z_{x}(x,t)^2 {\rm d}x 	. 
\end{align}
Taking the time derivative of \eqref{V2def} along the solution of \eqref{zPDE}--\eqref{zBC2} leads to (note that \eqref{zBC2} yields $z_t(s(t),t) = - \dot{s}(t) z_x(s(t),t)$), 
\begin{align} 
	\dot{ \hat{V}}_2 = & \frac{\dot{s}(t)}{2} \hat z_{x}(s(t),t)^2 + \int_0^{s(t)} \hat z_{x}(x,t) \hat z_{xt}(x,t) {\rm d}x \notag\\
	= & \frac{\dot{s}(t)}{2}\hat z_{x}(s(t),t)^2 + \hat z_{x}(s(t),t) \hat z_{t}(s(t),t) - \hat z_{x}(0,t) \hat z_{t}(0,t) \notag\\
	& - \alp_{{\rm s}} || \hat z_{xx}||^2   - \lambda  || \hat z_{x}||^2 \notag\\
	& + \dot{s}(t) X(t)\int_0^{s(t)} \hat z_{xx}(x,t) e^{-\gm x}  g(x-s(t))  {\rm d}x \notag\\
	&- \bar{\beta}  \tilde{z}_{x}(s(t),t) \int_0^{s(t)} \hat z_{xx}(x,t)\phi(x-s(t)) {\rm d}x, \notag\\
		= & -\frac{\dot{s}(t)}{2} \hat z_{x}(s(t),t)^2  - \alp_{{\rm s}} || \hat z_{xx}||^2   - \lambda  || \hat z_{x}||^2 \notag\\
	& + \dot{s}(t) X(t) \left(\hat z_{x}(s(t),t) g(0) e^{-\gm s(t)} \right. \notag\\
	& \left.  - \int_0^{s(t)} \hat z_{x}(x,t) ( g'(x-s(t)) - \gm g)e^{-\gm x}   {\rm d}x \right) \notag\\
	&- \bar{\beta}  \tilde{z}_{x}(s(t),t) \int_0^{s(t)} \hat z_{xx}(x,t) \phi(x-s(t)) {\rm d}x.	\label{V2der}
\end{align}
By Cauchy-Schwarz and Young's inequalities, for $\delta_1>0$, it holds that 
\begin{align} 
&g(0) X(t) \hat z_{x}(s(t),t)e^{-\gm s(t)} \leq \frac{g(0)^2}{2} X(t)^2 + \frac{1}{2} \hat z_{x}(s(t),t)^2, \label{CSY1} \\
&-\bar{\beta}  \tilde{z}_{x}(s(t),t) \int_0^{s(t)} \hat z_{xx}(x,t)  \phi(x-s(t)) {\rm d}x\notag\\
& \leq \frac{\delta_1}{2} \bar{\beta}^2  \tilde{z}_{x}(s(t),t)^2 + \frac{1}{2 \delta_1} || \phi ||_{L_2}^2 || \hat z_{xx}||_{L_2}. \label{CSY2}
\end{align} 
Applying \eqref{CSY1} and \eqref{CSY2} to \eqref{V2der} with setting $\delta_1 = \frac{2 || \phi ||_{L_2}^2}{\alpha_{{\rm s}}} $ yields
\begin{align} 
\dot{ \hat{V}}_2 \leq  &  - \frac{\alp_{{\rm s}}}{2} || \hat z_{xx}||^2   - \lambda || \hat z_{x}||^2 \notag\\
	& + \frac{\dot{s}(t)}{2} \left( \bar g X(t)^2 + ||\hat z_{x}||^2 \right) + \frac{ \bar{\beta}^2  || \phi ||_{L_2}^2}{\alpha_{{\rm s}}}  \tilde{z}_{x}(s(t),t)^2  , \label{V2fin} 
\end{align} 
where $\bar g : = \max_{s(t) \in (0, s_{{\rm r}})} (g(0)^2 + g(-s(t))^2 + || g'||^2 ) $. Let $\hat V_3$ be Lyapunov functional defined by 
\begin{align} 
\hat V_3 = \frac{1}{2} X(t)^2 . \label{V3def} 
\end{align} 
Taking the time derivative of \eqref{V3def} together with \eqref{zODE} leads to 
\begin{align} 
\dot{\hat{V}}_3 &= - c X(t)^2  - \bar{\beta}\hat z_{x}(s(t),t) e^{\gm s(t)} X(t)+ \bar{\beta} \tilde z_{x}(s(t),t) e^{\gm s(t)}X(t). \label{V3timeder}
\end{align} 
Applying Young's and Agmon's inequalities to \eqref{V3timeder}, we get 
\begin{align} 
\dot{\hat{V}}_3 &\leq - \frac{c}{2} X(t)^2 + \frac{\bar \beta^2 e^{2 \gm s_{{\rm r}}} }{c} \hat z_{x}(s(t),t) ^2 + \frac{\bar \beta^2 e^{2\gm s_{{\rm r}}} }{ c} \tilde z_{x} (s(t),t)^2, \notag\\
&\leq - \frac{c}{2} X(t)^2 + \frac{4 \bar \beta^2 s_{{\rm r}}  e^{2 \gm s_{{\rm r}}} }{c} || \hat z_{xx}||^2 + \frac{\bar \beta^2 e^{2\gm s_{{\rm r}}} }{ c} \tilde z_{x} (s(t),t)^2. \label{V3fin}
\end{align} 
Consider the following functional  
\begin{align}\label{H1norm}
\hat V = \hat V_1 + \hat V_2 + p \hat V_3  , 
\end{align}
where $p = \frac{c\alpha_{{\rm s}} e^{-2 \gm s_{{\rm r}}}}{16 \bar \beta^2 s_{{\rm r}}}$. Combining \eqref{V1fin}, \eqref{V2fin}, and \eqref{V3fin}, the time derivative of \eqref{H1norm} is shown to satisfy the following inequality 
\begin{align}
\dot{\hat V} \leq &  - \frac{\alp_{{\rm s}}}{4} || \hat z_{xx}||^2 - (\alp_{{\rm s}} + \lambda)  || \hat z_{x}||^2  - \lambda  || \hat z||^2  - \frac{pc}{2} X(t)^2 \notag\\
	& +\frac{\dot{s}(t)}{2} \left( (1 + \bar g) X(t)^2 + || g||^2 || \hat z ||^2 + || \hat z_{x}||^2  \right) \notag\\
	& +\left( \bar{\beta}^2 || \phi ||^2 \left( \frac{ 1}{2 h_{{\rm s}}} + \frac{1}{\alpha_{{\rm s}}} \right)  +  \frac{\alpha_{{\rm s}}}{16 s_{{\rm r}}}  \right) \tilde{z}_{x}(s(t),t)^2, \notag\\
	\leq &  - \frac{\alp_{{\rm s}}}{4} || \hat z_{xx}||^2 - (\alp_{{\rm s}} + \lambda)  || \hat z_{x}||^2  - \lambda  || \hat z||^2  - \frac{pc}{2} X(t)^2 \notag\\
	& +a \dot{s}(t) \hat V  +M_1 ||\tilde{z}_{xx}||^2, \label{Vdot}
\end{align}
where $a = \max\{\frac{(1 + \bar g)}{p}, ||g||^2, 1\}$, $M_1 = 4 s_{{\rm r}}\left( \bar{\beta}^2 || \phi ||^2 \left( \frac{ 1}{2 h_{{\rm s}}} + \frac{1}{\alpha_{{\rm s}}} \right)  +  \frac{\alpha_{{\rm s}}}{16 s_{{\rm r}}}  \right)$. Thus, using the Lyapunov function $\tilde V$ in \eqref{defV1tilde} for the estimation error $\tilde z$-system \eqref{ztilde1}--\eqref{ztilde3}, we define the Lyapunov function for the total $(\hat z, X, \tilde z)$-system as 
\begin{align} 
V = \hat V + \frac{2 M_1}{\alpha_{{\rm s}}} \tilde V. 	
\end{align}
Then, by combining the inequalities \eqref{Vtildeineq} and \eqref{Vdot}, we arrive at 
\begin{align} 
\dot V \leq &	 - \frac{\alp_{{\rm s}}}{4} \left( || \hat z_{xx}||^2 +\frac{2 M_1}{\alpha_{{\rm s}}} || \tilde z_{xx}||^2 \right)  \notag\\
&- (\alp_{{\rm s}} + \lambda) \left( || \hat z_{x}||^2 +\frac{2 M_1}{\alpha_{{\rm s}}} || \tilde z_{x}||^2 \right) \notag\\
&  - \lambda \left( || \hat z||^2 +\frac{2 M_1}{\alpha_{{\rm s}}} || \tilde z||^2 \right)    - \frac{pc}{2} X(t)^2  +a \dot{s}(t) \hat V , \notag\\
\leq &	   - \left( \frac{\alp_{{\rm s}}}{16 s_{{\rm r}}} + \lambda \right) \left( || \hat z_{x}||^2 +|| \hat z||^2 +\frac{2 M_1}{\alpha_{{\rm s}}} || \tilde z_{x}||^2 + \frac{2 M_1}{\alpha_{{\rm s}}} || \tilde z||^2\right) \notag\\
&     - \frac{pc}{2} X(t)^2  +a \dot{s}(t) \hat V , \notag\\
\leq & - d V + a \dot{s}(t) V , \label{Vineq}
\end{align}
where 
\begin{align} 
d = \min\left\{ \frac{\alp_{{\rm s}}}{16 s_{{\rm r}}} + \lambda,	  c\right\}. 
\end{align}
Following the procedure in \cite{koga2017backstepping}, the inequality \eqref{Vineq} with \eqref{property2} and \eqref{position} leads to the exponential norm estimate 
\begin{align} 
V(t) \leq e^{a (s(t) - s_0 ) } V(0) e^{- dt} \leq e^{a s_{{\rm r}} } V(0) e^{- dt}. 
\end{align}
Let $\Psi (t) =||w||_{{\cal H}_1}^2 + X(t)^2$. Then, we have $\underline{M} V \leq  \Psi (t) \leq \bar{M} V$ where $ \bar{M} = 2\max \{ e^{2 \gm s_{{\rm r}} } (1+ \gm^2) , \fr{1}{p} \}$, $\underline{M} = \left(\max \{ 2 (1+ \gm^2), \fr{p}{2} \}\right)^{-1} $. Therefore, $\Psi(t) \leq \fr{\bar{M}}{\underline{M}}  e^{a s_{{\rm r}} } \Psi(0) e^{- dt}$, which proves the exponential stability of the target $w$-system in ${\cal H}_1$-norm. Since the $u$-system in \eqref{err1}-\eqref{err3} and the target $w$-system in \eqref{tarPDE}-\eqref{tarODE} have equivalent stability property due to the invertibility of the backstepping transformation \eqref{bkst}, the exponential estimate in ${\cal H}_1$-norm is also guaranteed for the $u$-system, which concludes  the proof of Theorem \ref{theo}.

\section{Simulation}\label{simulation}
\begin{table}
\centering 
\vspace{5mm}

\caption{HDPE parameters obtained by \cite{tadmor2006principles}.}
\begin{tabular}{lll}
\hline
melting point & $T_{{\rm m}}$& $\unit[135]{^{\circ}C}$\\
specific heat solid & $c_{{\rm s}}$ & $\unit[1895]{J kg^{-1}K^{-1}}$\\
specific heat melt & $c_{{\rm l}}$ & $\unit[2640]{J kg^{-1}K^{-1}}$\\
therm. conduct. solid & $k_{{\rm s}}$ &$\unit[0.373]{Wm^{-1}K^{-1}}$\\
therm. conduct. melt & $k_{{\rm l}}$ &$\unit[0.324]{Wm^{-1}K^{-1}}$\\
solid density & $\rho_{{\rm s}}$ & $\unit[955]{kgm^{-3}}$\\
melt density & $\rho_{{\rm l}}$ & $\unit[780]{kgm^{-3}}$\\
heat of fusion&$\Delta H$ &$\unit[39000]{J kg^{-1}}$\\
\hline
\end{tabular}
\label{table:tadmorparam}
\end{table}

\begin{figure}[t]
\begin{center}
\includegraphics[width=7.0cm]{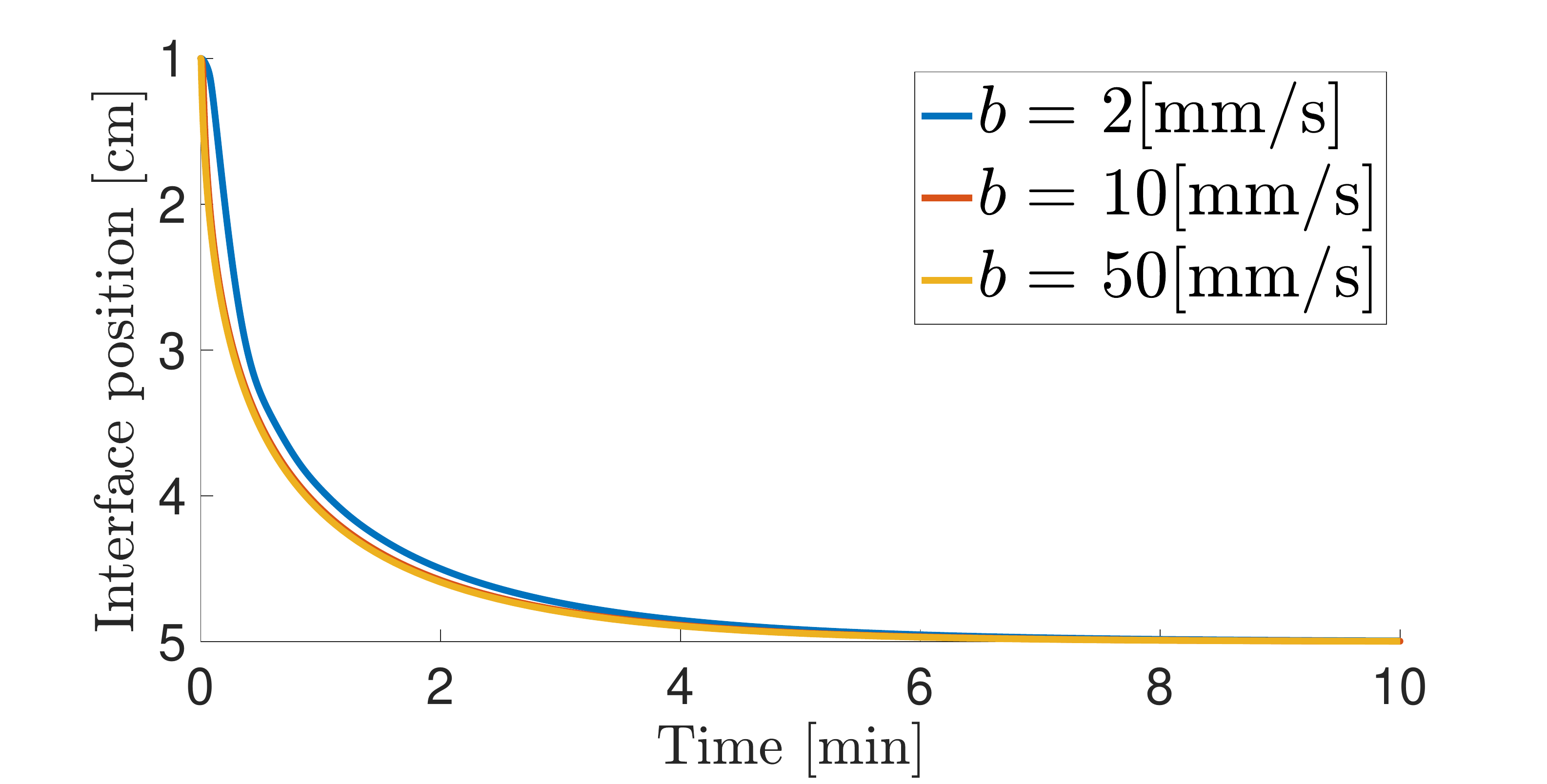}\\
\caption{The closed-loop response of the interface position. For each operating speed, the interface position is stabilized after 6[min]. }
\label{fig:interface}
\includegraphics[width=7.0cm]{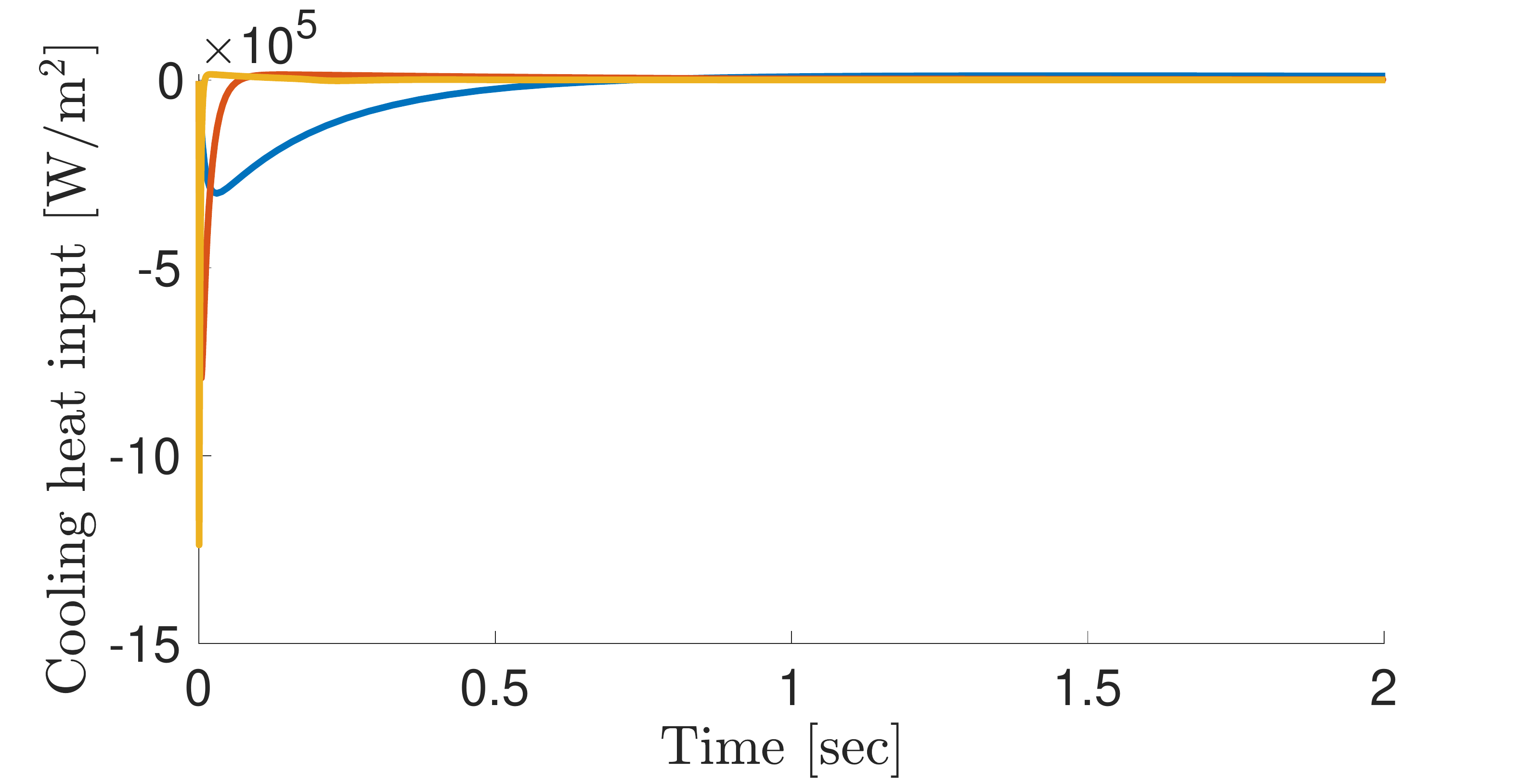}\\
\caption{The closed-loop response of the output feedback control. The transient gets shorter as the operation gets faster.}
\label{fig:control}
\includegraphics[width=7.0cm]{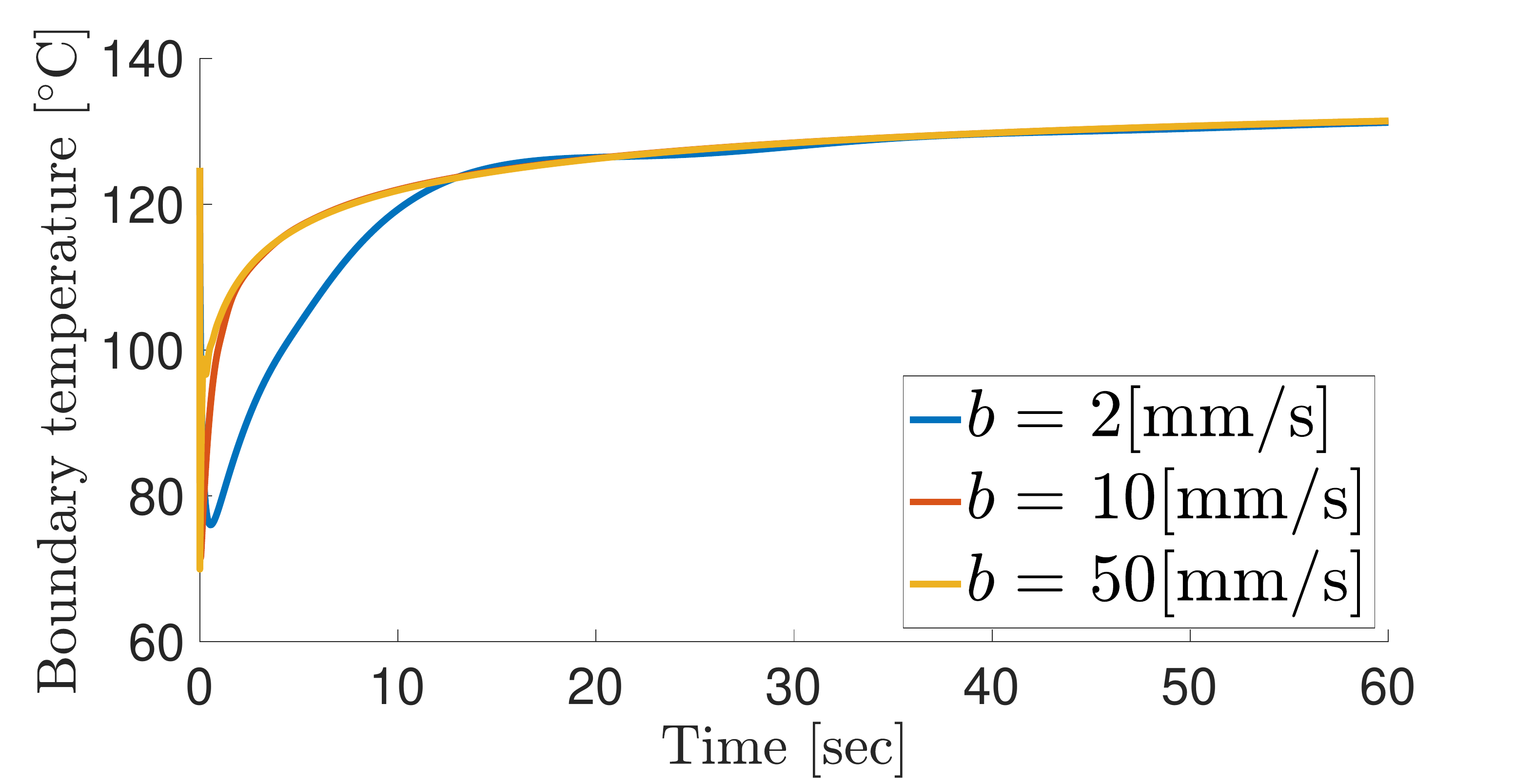}\\
\caption{The response of the boundary temperature, which maintains reasonable value for the material and safe operation. }
\label{fig:temperature}
\end{center}
\end{figure}

\begin{figure}[t]
\begin{center}
\includegraphics[width=7.0cm]{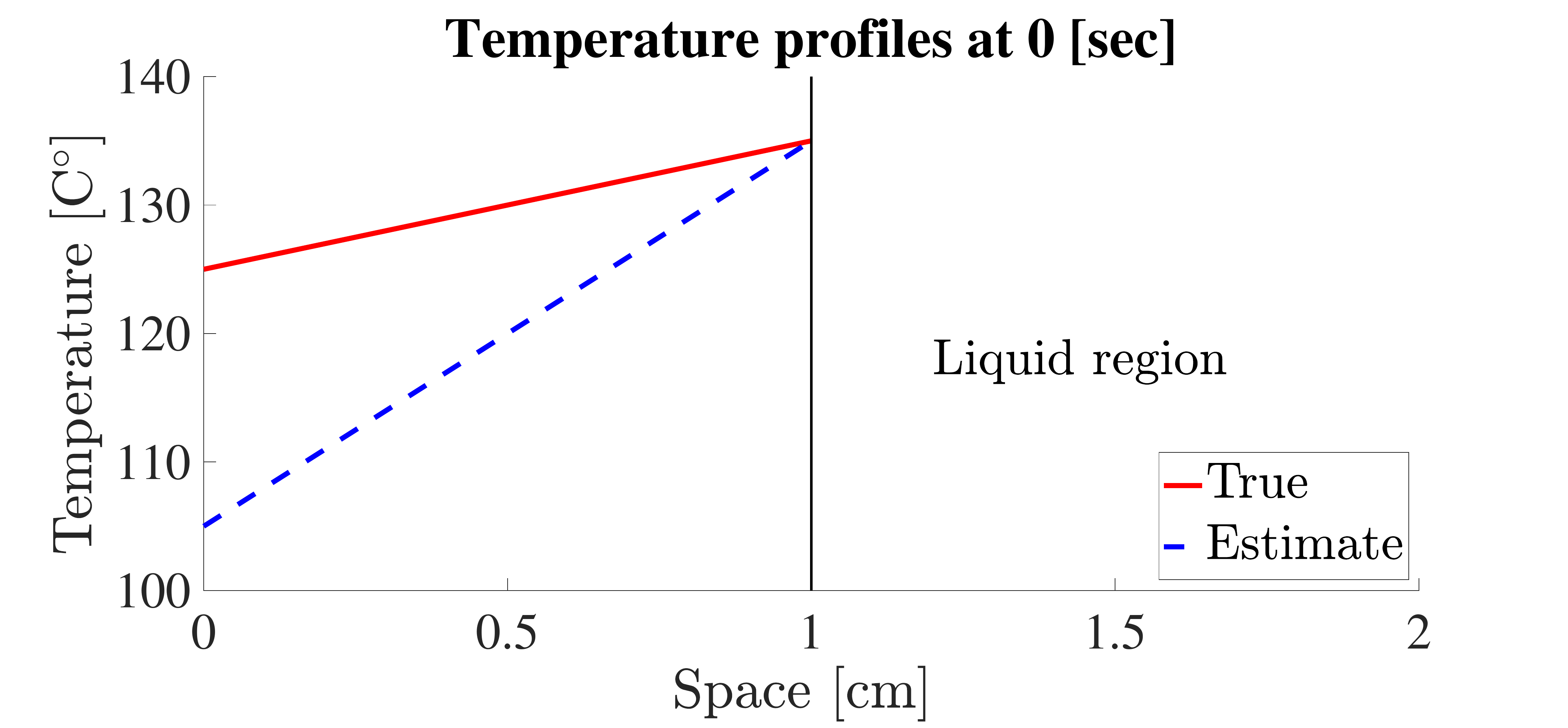}\\
\includegraphics[width=7.0cm]{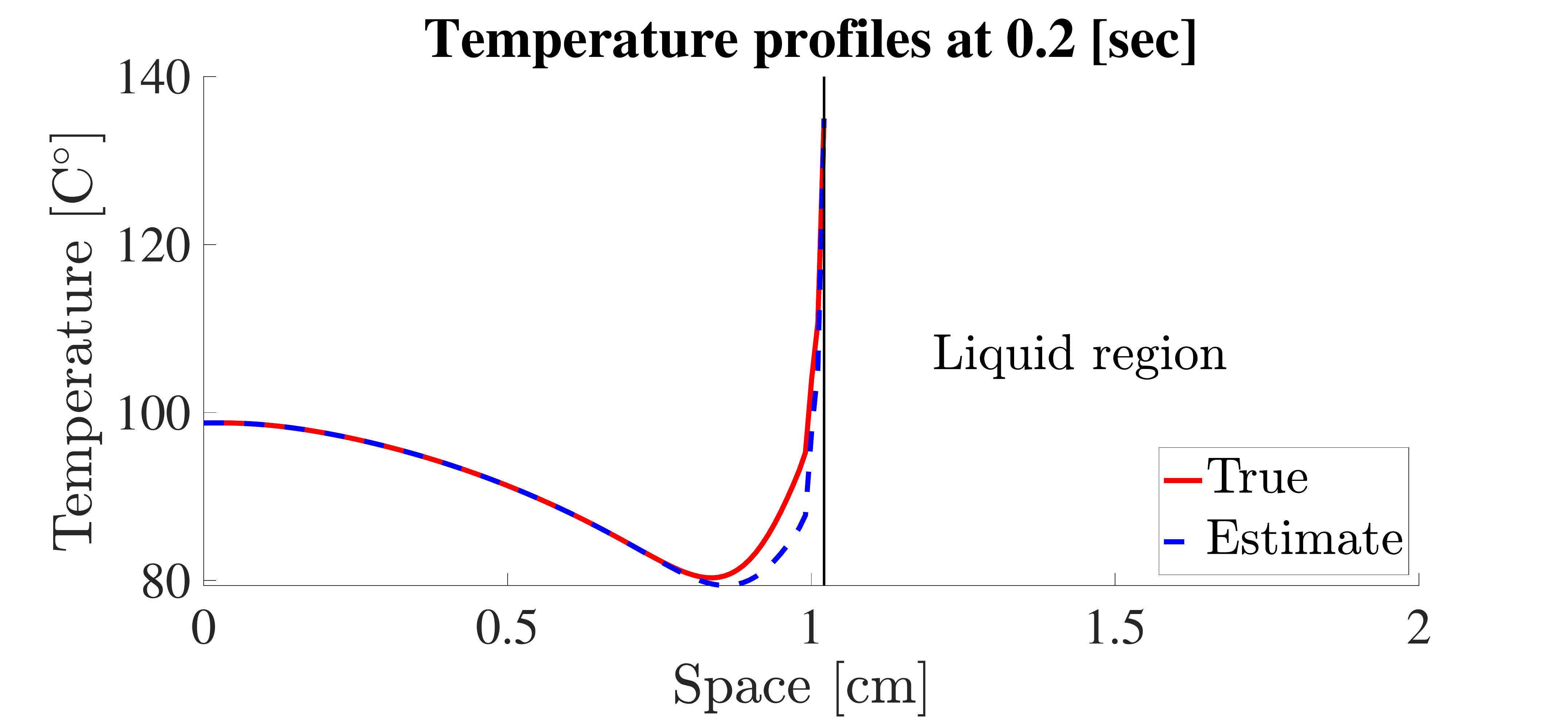}\\
\includegraphics[width=7.0cm]{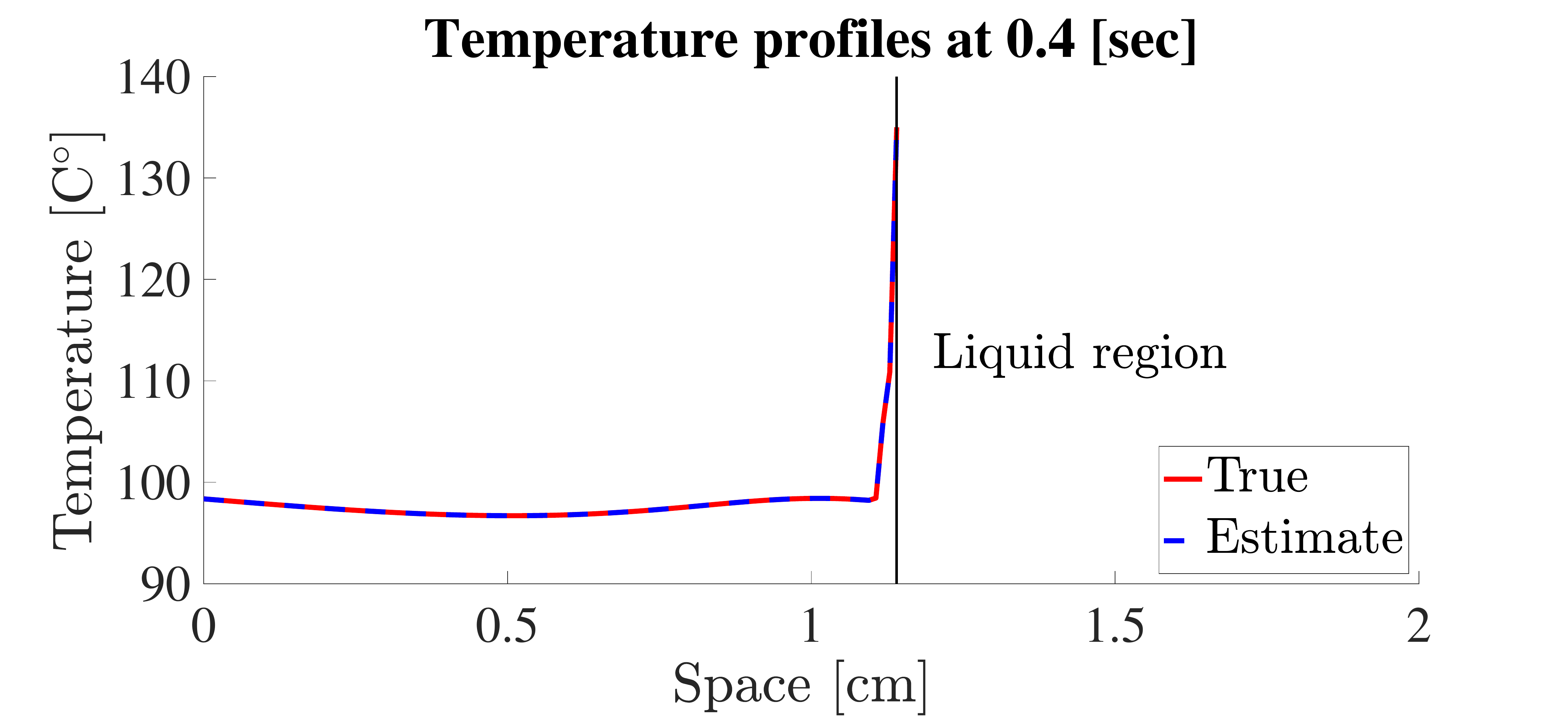}\\
\caption{The comparison of the true and estimated temperature profiles.}
\label{fig:estimate}
\end{center}
\end{figure}

For numerical study to investigate the controller's performance in different operating conditions, we have employed the simulation of the original "two-phase" model governed by \eqref{sys1}--\eqref{sys5} without assuming that the liquid phase is at stead-state, run the PDE observer given in \eqref{obs1}--\eqref{obs3}, and implemented the associated output feedback controller \eqref{qfcont}. We used the boundary immobilization method to obtain a fixed
boundary system and discretized the system with finite differnces to construct  a finite dimensional representation of the model. Using Matlab's ode23s solver, we simulated the setup with three different advection speeds ranging from \unit[2]{mm/s} to \unit[50]{mm/s}, to cover a wide spectrum of operating modes. The material parameters  are chosen from
\cite{tadmor2006principles},  in which distinct
values for high density polyethylene in solid and liquid state were
experimentally derived (see Table~\ref{table:tadmorparam}). The extruder length is $L=\unit[10]{cm}$ and the
constant barrel temperature is set to $T_{{\rm b}} = \unit[145]{^\circ C}$. Further, the auxiliary heat input was chosen $q_{\rm{m}}^* = \unitfrac[100]{W}{m^2}$. The initial
temperature is set as a linear profile, which satisfies the boundary
condition~\eqref{sys4}. For each advection speed, the control parameter $c$ is adjusted to generate a
reasonable boundary temperature. While higher advection
speeds enable shorter convergence times, it is only achieved when the control
gain $c$ is also sufficiently large. However, larger values in the control
gain will initially produce very low boundary temperature values. For each configuration, we adjusted the control gain as $c=$0.2 for $b=$2[mm/s], $c=$1.0 for $b=$10[mm/s], and $c=$5.0 for $b=$50[mm/s], to produce comparable temperature peaks at the inlet and prevent unrealistic values. The closed-loop responses of the interface position $s(t)$, the boundary control input $q_{{\rm f}}(t)$, and the boundary temperature $T_{{\rm s}}(0,t)$ are shown in Fig.~\ref{fig:interface}, Fig.~\ref{fig:control} and Fig.~\ref{fig:temperature}, respectively. The interface responses, depicted in Fig.~\ref{fig:interface}, have quite  similar behaviors in all three setups. However, the control input, shown in Fig.~\ref{fig:control}, appears to act faster for higher advection speeds but exhibits a similar qualitative behavior. Similar properties were observed in the boundary temperature response in Fig.\ref{fig:temperature}. Note that all the three figures have different time ranges.

Moreover, for the fast operating condition $b=$50[mm/s], the comparison of the estimated temperature profile and the true temperature profile at $t=$0[sec], 0.2[sec], 0.4[sec] are shown in Fig. \ref{fig:estimate}. We can observe that the estimated temperature profile gets almost same as the true  temperature profile at 0.4[sec], associated with the expansion of the solid granules' region. Hence, the convergence of the designed observer to the true temperature profile is approximately 1000 times faster than the convergence of the interface position to the setpoint position, which is sufficiently quick performance of the temperature estimation. 

For comparison, we also tested a closed loop setup with PI control given by
\begin{equation}
q_{\rm{f}}(t) = q_{\rm{f}}^* + K_{\rm{P}}(s(t)-s_{\rm{r}}) + K_{\rm{I}}\int_{t_0}^t(s(\tau)-s_{\rm{r}}) d\tau,   
\end{equation} 
where $K_{\rm{P}}$ and $K_{\rm{I}}$ are gain parameters to be tuned in order to achieve the desired performance. However, for any choice of the parameters we have tried, the closed-loop response of the interface position does not stabilize at the sepoint $s_{{\rm r}}$. Fig. \ref{fig:PI} depicts the responses under PI control with a relatively suitable choice of the gains. The lower plot in Fig. \ref{fig:PI} shows that the temperature at the inlet of the extruder gets above the melting temperature at approximately 2.9 [min], which violates the validity condition~\eqref{valid2} of the solid polymer temperature, while our proposed output feedback control guarantees to satisfy the condition under the closed-loop system. Such an overshoot behavior beyond the melting temperature might be reduced by PID control, however, the velocity of the interface position is nearly impossible to measure online and the differentiator generally causes high noise. Overall, the proposed output feedback control law illustrates superior performance to PI control in terms of both convergence to the setpoint and the validity condition. 

From the simulations, we conclude that our control design achieves a stable interface position, even with very fast
advection speeds \unit[50]{mm/s}, with which a particle inserted
in the inlet will travel in two seconds through the extruder, when
assuming a \unit[10]{cm} extruder.

\begin{figure}[t]
\begin{center}
\includegraphics[width=8.0cm]{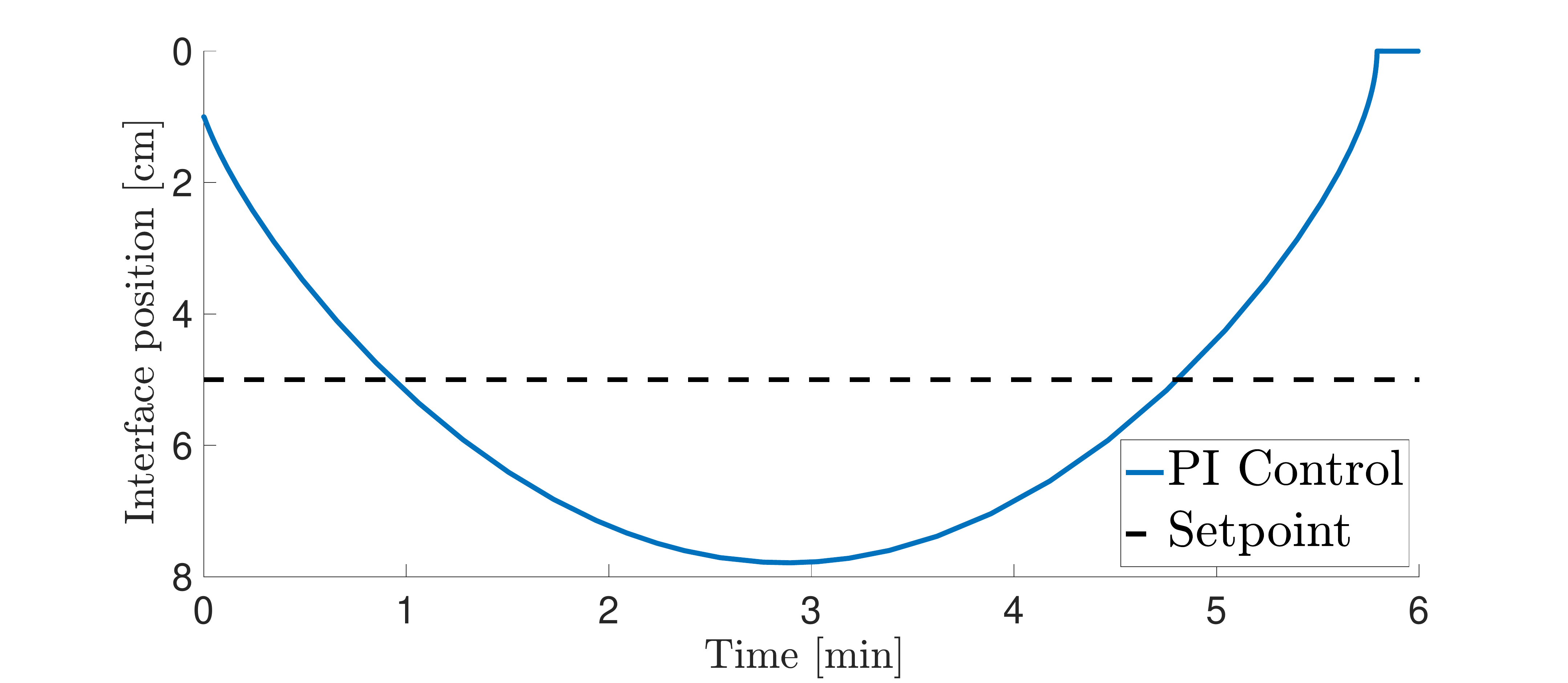}\\
\includegraphics[width=8.0cm]{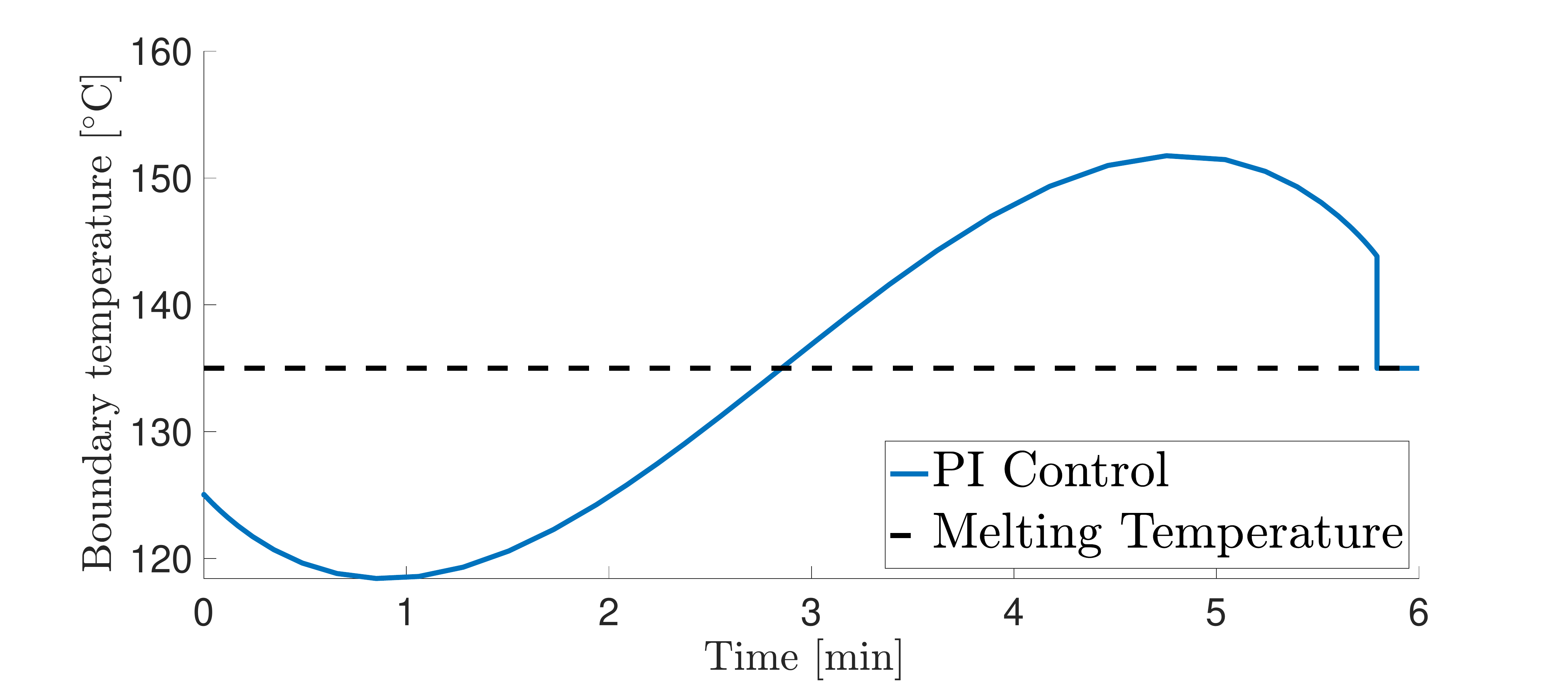}\\
\caption{The response with PI control. The boundary temperature is heated up and gets above the melting temperature, which violates the condition for the solid phase temperature. }
\label{fig:PI}
\end{center}
\end{figure}

\section{Conclusions}\label{conclusion}
In this paper, we designed an observer and the associated output feedback control to stabilize an ink production process of the screw extrusion based 3D printing. The steady-state analysis is provided by setting the setpoint as a , and the control design to stabilize the interface position is derived. The simulation results prove the effectiveness of the boundary feedback control law for some given screw speeds. 
\bibliographystyle{plain}

\end{document}